\documentclass{amsart}

\usepackage{amssymb,upref}
\usepackage[mathcal]{euscript}
\usepackage[cmtip,all]{xy}

\newcommand{\bydef}{:=}
\newcommand{\defby}{=:}
\DeclareMathOperator{\CD}{\mathfrak{CD}}
\DeclareMathOperator{\Cl}{\mathfrak{Cl}}
\newcommand{\bi}{\mathbf{i}}

\newcommand{\sym}{\mathcal{H}}
\newcommand{\sg}{\mathrm{Sym}}

\newcommand{\cA}{\mathcal{A}}
\newcommand{\calA}{\mathcal{A}}
\newcommand{\cB}{{\mathcal B}}

\newcommand{\cC}{\mathcal{C}}
\newcommand{\cD}{{\mathcal D}}
\newcommand{\cQ}{\mathcal{Q}}
\newcommand{\calQ}{\mathcal{Q}}
\newcommand{\cK}{\mathcal{K}}
\newcommand{\calC}{\mathcal{C}}

\newcommand{\cM}{{\mathcal M}}
\newcommand{\cR}{\mathcal{R}}

\newcommand{\calK}{{\mathcal K}}

\newcommand{\W}{W} 

\newcommand{\id}{{\mathrm{id}}} 

\newcommand{\frg}{{\mathfrak g}}

\newcommand{\ul}[1]{\underline{#1}}
\newcommand{\wh}[1]{\widehat{#1}}
\newcommand{\wt}[1]{\widetilde{#1}}

\newcommand{\vphi}{\varphi}
\newcommand{\veps}{\varepsilon}

\newcommand{\diag}{\mathrm{diag}}

\DeclareMathOperator*{\ot}{\otimes}

\newcommand{\ZZ}{\mathbb{Z}}
\newcommand{\bZ}{{\mathbb Z}}

\newcommand{\bR}{{\mathbb R}}

\newcommand{\FF}{\mathbb{F}}
\newcommand{\bF}{{\mathbb F}}
\newcommand{\chr}[1]{\mathrm{char}\,#1}

\DeclareMathOperator{\Hom}{\mathrm{Hom}}
\DeclareMathOperator{\End}{\mathrm{End}}

\DeclareMathOperator{\Aut}{\mathrm{Aut}}
\DeclareMathOperator{\Stab}{\mathrm{Stab}}
\DeclareMathOperator{\Diag}{\mathrm{Diag}}
\DeclareMathOperator{\Der}{\mathrm{Der}}

\DeclareMathOperator{\supp}{\mathrm{Supp}\,}
\DeclareMathOperator{\Supp}{\mathrm{Supp}\,}

\newcommand{\Ad}{\mathrm{Ad}}

\newcommand{\Spin}{\mathrm{Spin}}

\newcommand{\M}{\Gamma_\cM}

\DeclareMathOperator{\degree}{deg}

\newcommand{\subo}{_{\bar 0}}
\newcommand{\subuno}{_{\bar 1}}

\newtheorem{theorem}{Theorem}[section]
\newtheorem{proposition}[theorem]{Proposition}

\newtheorem{corollary}[theorem]{Corollary}

\theoremstyle{definition}
\newtheorem{df}[theorem]{Definition}

\theoremstyle{remark}
\newtheorem{remark}[theorem]{Remark}

\def\hregleta{\hrule height .5pt}
\def\hreglon{\hrule height1pt}
\def\vreglon{\vrule height 12pt width1pt depth 4pt}
\def\vregleta{\vrule width .5pt}
\def\hreglonfill{\leaders\hreglon\hfill}

\def\hregletafill{\leaders\hregleta\hfill}

\begin{document}

\title[Weyl groups of fine gradings]{Weyl groups of fine gradings on matrix algebras, octonions and the Albert algebra}

\author[Alberto Elduque]{Alberto Elduque$^{\star}$}
\thanks{$^{\star}$ Supported by the Spanish Ministerio de Educaci\'{o}n y Ciencia and
FEDER (MTM 2007-67884-C04-02) and by the Diputaci\'on General de Arag\'on (Grupo de Investigaci\'on de \'Algebra)}
\address{Departamento de Matem\'aticas e Instituto Universitario de Matem\'aticas y Aplicaciones,
Universidad de Zaragoza, 50009 Zaragoza, Spain}
\email{elduque@unizar.es}

\author[Mikhail Kochetov]{Mikhail Kochetov$^{\star\star}$}
\thanks{$^{\star\star}$Supported by the Natural Sciences and Engineering Research Council (NSERC) of Canada, Discovery Grant \# 341792-07.}
\address{Department of Mathematics and Statistics,
Memorial University of Newfoundland, St. John's, NL, A1C5S7, Canada}
\email{mikhail@mun.ca}




\begin{abstract}
Given a grading $\Gamma:\cA=\bigoplus_{g\in G}\cA_g$ on a nonassociative algebra $\cA$ by an abelian group $G$, we have two subgroups of $\Aut(\cA)$: the automorphisms that stabilize each component $\cA_g$ (as a subspace) and the automorphisms that permute the components. By the Weyl group of $\Gamma$ we mean the quotient of the latter subgroup by the former. In the case of a Cartan decomposition of a semisimple complex Lie algebra, this is the automorphism group of the root system, i.e., the so-called extended Weyl group. A grading is called fine if it cannot be refined. We compute the Weyl groups of all fine gradings on matrix algebras, octonions and the Albert algebra over an algebraically closed field (of characteristic different from $2$ in the case of the Albert algebra).
\end{abstract}

\maketitle


\section{Introduction}

Let $\cA$ be an algebra (not necessarily associative) over a field $\FF$ and let $G$ be a group.
We will usually use multiplicative notation for $G$, but switch to additive notation when working with groups such as $\ZZ$ and $\ZZ_m\bydef\ZZ/m\ZZ$.

\begin{df}\label{df:G_graded_alg}
A {\em $G$-grading} on $\cA$ is a vector space decomposition
\[
\Gamma:\;\cA=\bigoplus_{g\in G} \cA_g
\]
such that
\[
\cA_g \cA_h\subset \cA_{gh}\quad\mbox{for all}\quad g,h\in G.
\]
If such a decomposition is fixed, we will refer to $\cA$ as a {\em $G$-graded algebra}.
The nonzero elements $a\in\cA_g$ are said to be {\em homogeneous of degree $g$}; we will write $\deg a=g$.
The {\em support} of $\Gamma$ is the set $\supp\Gamma\bydef\{g\in G\;|\;\cA_g\neq 0\}$.
\end{df}

There are two natural ways to define equivalence relation on group graded algebras. We will use the term ``isomorphism'' for the case when the grading group is a part of definition  and ``equivalence'' for the case when the grading group plays a secondary role. Let
\[
\Gamma:\; \cA=\bigoplus_{g\in G} \cA_g\mbox{ and }\Gamma':\;\cB=\bigoplus_{h\in H} \cB_h
\]
be two gradings on algebras, with supports $S$ and $T$, respectively.

\begin{df}\label{df:equ_grad}
We say that $\Gamma$ and $\Gamma'$ are {\em equivalent} if there exists an isomorphism of algebras $\vphi\colon\cA\to\cB$ and a bijection $\alpha\colon S\to T$ such that $\varphi(\cA_s)=\cB_{\alpha(s)}$ for all $s\in S$. Any such $\vphi$ will be called an {\em equivalence} of $\Gamma$ and $\Gamma'$. If the gradings on $\cA$ and $\cB$ are clear from the context, then we will say that $\cA$ and $\cB$ are {\em equivalent}.
\end{df}

The algebras graded by a fixed group $G$ form a category where the morphisms are the {\em homomorphisms of $G$-graded algebras}, i.e., homomorphisms of algebras $\varphi\colon\cA\to\cB$ such that $\vphi(\cA_g)\subset\cB_g$ for all $g\in G$.

\begin{df}\label{df:iso_grad}
In the case $G=H$, we say that $\Gamma$ and $\Gamma'$ are {\em isomorphic} if $\cA$ and $\cB$ are isomorphic as $G$-graded algebras, i.e., there exists an isomorphism of algebras $\varphi\colon\cA\to\cB$ such that $\varphi(\cA_g)=\cB_{g}$ for all $g\in G$.
\end{df}

Replacing $G$ with a smaller group, we can assume that $\supp\Gamma$ generates $G$. Even with this assumption, there will be, in general, many other groups $H$ such that the vector space decomposition $\Gamma$ can be realized as an $H$-grading. It turns out \cite{PZ} that there is one distinguished group among them.

\begin{df}\label{df:univ_group}
Suppose that $\Gamma$ admits a realization as a $G_0$-grading for some group $G_0$. We will say that $G_0$ is a {\em universal group of $\Gamma$} if for any other realization of $\Gamma$ as a $G$-grading, there exists a unique homomorphism $G_0\to G$ that restricts to identity on $\supp\Gamma$. We define the {\em universal abelian group} in the same manner.
\end{df}

Note that, by definition, $G_0$ is a group with a distinguished generating set, $\supp\Gamma$.
A standard argument shows that, if a universal (abelian) group exists, it is unique up to an isomorphism over $\supp\Gamma$.
We will denote it by $U(\Gamma)$. It turns out that $U(\Gamma)$ exists and depends only on the equivalence
class of $\Gamma$. Details may be found in \cite{Ksur}.

Following \cite{PZ}, we can associate three subgroups of $\Aut(\cA)$ to a grading $\Gamma$ on an algebra $\cA$.

\begin{df}\label{df:aut_diag}
The {\em automorphism group of $\Gamma$}, denoted $\Aut(\Gamma)$, consists of all automorphisms of $\cA$ that permute the components of $\Gamma$. Each $\vphi\in\Aut(\Gamma)$ determines a self-bijection $\alpha=\alpha(\varphi)$ of the support $S$ such that $\varphi(\cA_s)=\cA_{\alpha(s)}$ for all $s\in S$.
The {\em stabilizer of $\Gamma$}, denoted $\Stab(\Gamma)$, is the kernel of the homomorphism $\Aut(\Gamma)\to\sg(S)$ given by $\vphi\mapsto\alpha(\vphi)$. Finally, the {\em diagonal group of $\Gamma$}, denoted $\Diag(\Gamma)$, is the subgroup of the stabilizer consisting of all automorphisms $\varphi$ such that the restriction of $\varphi$ to any homogeneous component of $\Gamma$ is the multiplication by a (nonzero) scalar.
\end{df}

Thus $\Aut(\Gamma)$ is the group of self-equivalences of the graded algebra $\cA$ and $\Stab(\Gamma)$ is the group of automorphisms of the graded algebra $\cA$.  $\Diag(\Gamma)$ is isomorphic to the group of characters of $U(\Gamma)$. If $\dim\cA<\infty$, then $\Diag(\Gamma)$ is a diagonalizable algebraic group (quasitorus). If, in addition, $\FF$ is an algebraically closed field of characteristic $0$ and $G$ is abelian, then $\Gamma$ is the eigenspace decomposition of $\cA$ relative to $\Diag(\Gamma)$ --- see e.g. \cite{Ksur}, the group $\Stab(\Gamma)$ is the centralizer of $\Diag(\Gamma)$, and $\Aut(\Gamma)$ is its normalizer.

\begin{df}
The quotient group $\Aut(\Gamma)/\Stab(\Gamma)$, which is a subgroup of $\sg(S)$, will be called the {\em Weyl group of $\Gamma$} and denoted by $\W(\Gamma)$.
\end{df}

We use the term ``Weyl group'', because if $\Gamma$ is the Cartan grading on a semisimple complex Lie algebra $\frg$, then $\W(\Gamma)$ is isomorphic to the so-called extended Weyl group of $\frg$, i.e., the automorphism group of the root system of $\frg$.

It follows from the universal property of $U(\Gamma)$ that any $\varphi\in\Aut(\Gamma)$ gives rise to a unique automorphism $u(\varphi)$ of $U(\Gamma)$ such that the following diagram commutes:
\[
\xymatrix{
{\supp\Gamma}\ar[r]\ar[d]_{\alpha(\varphi)} & U(\Gamma)\ar[d]^{u(\varphi)}\\
{\supp\Gamma}\ar[r] & U(\Gamma)
}
\]
where the horizontal arrows are the canonical imbeddings. This gives an action of $\Aut(\Gamma)$ by automorphisms of the group $U(\Gamma)$.
The kernel of this action is $\Stab(\Gamma)$, so we may regard $W(\Gamma)=\Aut(\Gamma)/\Stab(\Gamma)$ as a subgroup of $\Aut(U(\Gamma))$.

A grading $\Gamma$ is said to be {\em fine} if it cannot be refined in the class of (abelian) group gradings. Any $G$-grading on a finite-dimensional algebra $\cA$ is induced from some fine grading $\Gamma$ by a homomorphism $\alpha\colon U(\Gamma)\to G$ as follows: $\cA_g=\bigoplus_{u\in\alpha^{-1}(g)}\cA_u$ for all $g\in G$.

From now on, we assume that all grading groups are {\em abelian} and the ground field $\FF$ is {\em algebraically closed}. A description of fine gradings on the matrix algebras $M_n(\FF)$ was obtained in \cite{HPP98,BSZ} for characteristic $0$ and extended to characteristic $p$ in \cite{BZ03}. All fine gradings on the octonion algebra $\cC$ were described in \cite{E98}. All fine gradings on the Albert algebra $\cA$ in characteristic $0$ were classified in \cite{DM_f4}; the same classification was shown to be valid in characterisitc $p\neq 2$ in \cite{Albert}.

The Weyl groups were computed for some special cases of fine gradings on $M_n(\mathbb{C})$ in \cite{HPPT,PST,Han}. Here we compute them for all fine gradings on $M_n(\FF)$ over any algebraically closed field $\FF$ --- see Section \ref{se:matrix}. In Section \ref{se:octonions}, we compute the Weyl groups of fine gradings on the algebra of octonions. In Section \ref{se:Albert}, we compute them for the Albert algebra assuming that the characteristic is not $2$.


\section{Matrix algebras}\label{se:matrix}

The goal of this section is to compute the Weyl groups of all fine gradings on the matrix algebra $M_n(\FF)$. The proof will use graded modules over $M_n(\FF)$, so we will first state the classification of fine gradings in that language. All algebras in this section will be assumed {\em associative}. All algebras and modules will be assumed {\em finite-dimensional}.

\subsection{Gradings on matrix algebras}

Let $G$ be a group. A vector space $V$ is {\em $G$-graded} if it is equipped with a decomposition $V=\bigoplus_{g\in G}V_g$. A {\em graded subspace} $W\subset V$ is a subspace satisfying $W=\bigoplus_{g\in G}(V_g\cap W)$, so $W$ inherits a $G$-grading from $V$. We extend Definitions \ref{df:equ_grad} and \ref{df:iso_grad} to graded vector spaces. For $g\in G$, we define the {\em shift} $V^{[g]}$ of a $G$-graded vector space $V$ by setting $V^{[g]}_{hg}\bydef V_h$, $h\in G$.

Let $\cR$ be a $G$-graded algebra. A {\em graded left $\cR$-module} is a left $\cR$-module $V$ that is also a $G$-graded vector space such that
$\cR_g V_h\subset V_{gh}$ for all $g,h\in G$. A {\em graded right $\cR$-module} is defined similarly. A {\em homomorphism of graded $\cR$-modules} $f\colon V\to W$ is a homomorphism of $\cR$-modules such that $f(V_g)\subset W_g$ for all $g\in G$. 

We will follow the convention of writing homomorphisms of left modules on the right and homomorphisms of right modules on the left. Let $V$ and $W$ be graded left $\cR$-modules. Regarding $V$ and $W$ as $G$-graded vector spaces, we have the graded space $\Hom(V,W)=\bigoplus_{g\in G}\Hom_g(V,W)$ where
\[
\Hom_g(V,W)\bydef\{f\colon V\to W\;|\;(V_h)f\subset W_{hg}\;\mbox{ for all }\;h\in G\}.
\]
The space $\Hom_\cR(V,W)$ is a graded subspace in $\Hom(V,W)$. When $W=V$, we obtain a $G$-graded algebra $\End_\cR(V)\bydef\Hom_\cR(V,V)$.

A $G$-graded algebra $\cD$ is said to be a {\em graded division algebra} if it is unital and every nonzero homogeneous element has an inverse. Let $T\subset G$ be the support of $\cD$. Then $T$ is a subgroup of $G$. Any graded $\cD$-module $V$ is free and can be decomposed canonically into the direct sum of (nonzero) isotypical components:
\[
V=V_1\oplus\cdots\oplus V_s
\]
where $V_i$ is the sum of all graded submodules that are isomorphic to some fixed $\cD^{[g_i]}$, $g_i\in G$. The elements $g_1,\ldots,g_s$ are not uniquely determined, but their cosets $g_1 T,\ldots,g_s T$ are determined up to permutation. Write
\begin{equation}\label{df:gamma}
\gamma=(g_1,\ldots,g_s)\quad\mbox{where}\quad g_i^{-1}g_j\notin T\mbox{ for }i\neq j.
\end{equation}
If $\{v_1,\ldots,v_n\}$ is a homogeneous $\cD$-basis in $V$, then, for each $i$, the subset
\[
\{v_j\;|\;\deg v_j\in g_i T\}
\]
is a $\cD$-basis for $V_i$. Let $k_i=\dim_{\cD}V_i$ and write
\begin{equation}\label{df:kappa}
\kappa=(k_1,\ldots,k_s).
\end{equation}
Conversely, for a given pair $(\kappa,\gamma)$, let $V(G,\cD,\kappa,\gamma)$ be the right $\cD$-module that has a homogeneous $\cD$-basis consisting of $k_i$ elements of degree $g_i$, $i=1,\ldots,s$. Denote the $G$-graded algebra $\End_{\cD}(V)$ by $\cM(G,\cD,\kappa,\gamma)$.

We will say that a $G$-grading on $M_n(\FF)$ is a {\em division grading} if $M_n(\FF)$ is a graded division algebra. Such gradings can be constructed as follows. Suppose we have a finite subgroup $T\subset G$ such that there exists a {\em nondegenerate alternating bicharacter} $\beta\colon T\times T\to\FF^\times$, i.e., a function that is multiplicative in each variable and satisfies the following two properties: $\beta(t,t)=1$ for all $t\in T$ (alternating) and $\beta(u,T)=1$ implies $u=e$ (nondegenerate). Then $T$ admits a ``symplectic basis'', i.e., there exists a decomposition of $T$ into the direct product of cyclic subgroups:
\begin{equation}\label{decomp_T}
T=(H_1'\times H_1'')\times\cdots\times (H_r'\times H_r'')
\end{equation}
such that $H_i'\times H_i''$ and $H_j'\times H_j''$ are $\beta$-orthogonal for $i\neq j$, and $H_i'$ and $H_i''$ are in duality by $\beta$. Denote by $\ell_i$ the order of $H_i'$ and $H_i''$. (We may assume without loss of generality that $\ell_i$ are prime powers.) If we pick generators $a_i$ and $b_i$ for $H_i'$ and $H_i''$, respectively, then $\veps_i\bydef\beta(a_i,b_i)$ is a primitive $\ell_i$-th root of unity, and all other values of $\beta$ on the elements $a_1,b_1,\ldots,a_r,b_r$ are 1. Define the following elements of the algebra $M_{\ell_1}(\FF)\ot\cdots\ot M_{\ell_r}(\FF)$:
\begin{equation*}
X_{a_i}=I\ot\cdots I\ot X_i\ot I\ot\cdots I\quad\mbox{and}\quad X_{b_i}=I\ot\cdots I\ot Y_i\ot I\ot\cdots I,
\end{equation*}
where
\begin{equation*}
X_i=\begin{bmatrix}
\veps_i^{n-1} & 0                   & 0           & \ldots      & 0       & 0\\
0             & \veps_i^{n-2}       & 0           & \ldots      & 0       & 0\\
\ldots        &                     &             &             &         &  \\[3pt]
0             & 0                   & 0           & \ldots      & \veps_i & 0\\
0             & 0                   & 0           & \ldots      & 0       & 1
\end{bmatrix}\mbox{ and }
Y_i=\begin{bmatrix}
0 & 1 & 0 & \ldots & 0 & 0\\
0 & 0 & 1 & \ldots & 0 & 0\\
\ldots & & & & \\[3pt]
0 & 0 & 0 & \ldots & 0 & 1\\
1 & 0 & 0 & \ldots & 0 & 0
\end{bmatrix}
\end{equation*}
are the generalized Pauli matrices in the $i$-th factor, $M_{\ell_i}(\FF)$. Finally, set
\[
X_{(a_1^{i_1},b_1^{j_1},\ldots,a_r^{i_r},b_r^{j_r})}=X_{a_1}^{i_1}X_{b_1}^{j_1}\cdots X_{a_r}^{i_r}X_{b_r}^{j_r}.
\]
Identify $M_{\ell_1}(\FF)\ot\cdots\ot M_{\ell_r}(\FF)$ with $M_\ell(\FF)$, $\ell=\ell_1\cdots\ell_r$, via Kronecker product. Then
\begin{equation}\label{division_grading}
M_\ell(\FF)=\bigoplus_{t\in T}\FF X_t
\end{equation}
is a division grading with support $T$. Note that by construction we have
\begin{equation}\label{commutation_relations}
X_u X_v=\beta(u,v)X_v X_u\quad\mbox{for all}\quad u,v\in T.
\end{equation}

\begin{theorem}[\cite{BSZ,BZ03}]\label{class_division}
Let $T$ be a finite abelian group and let $\FF$ be an algebraically closed field. There exists a division grading on the matrix algebra $M_\ell(\FF)$ with support $T$ if and only if $\chr{\FF}$ does not divide $\ell$ and $T\cong\ZZ_{\ell_1}^2\times\cdots\times\ZZ_{\ell_r}^2$ where $\ell_1\cdots\ell_r=\ell$. All such gradings are equivalent to the grading given by \eqref{division_grading}.\hfill{$\square$}
\end{theorem}

\begin{theorem}[\cite{BSZ,BZ03}]\label{class_matrix}
Let $G$ be an abelian group and let $\cR$ be a matrix algebra with a $G$-grading. Then $\cR$ is isomorphic to some $\cM(G,\cD,\kappa,\gamma)$ where $\cD$ is a matrix algebra with division grading, and $\kappa$ and $\gamma$ are as in \eqref{df:gamma} and \eqref{df:kappa}.\hfill{$\square$}
\end{theorem}

The following notion and result will be crucial to our computation of Weyl groups.

\begin{df}
Let $G$ and $H$ be groups. Let $\cD$ be a $G$-graded algebra and $\cD'$ an $H$-graded algebra. Let $V$ be a graded right $\cD$-module and $V'$ a graded right $\cD'$-module.
An {\em equivalence from $(\cD,V)$ to $(\cD',V')$} is a pair $(\psi_0,\psi_1)$ such that $\psi_0\colon\cD\to\cD'$ is an equivalence of graded algebras and
$\psi_1\colon V\to V'$ is an equivalence of graded vector spaces, and $\psi_1(vd)=\psi_1(v)\psi_0(d)$ for all $v\in V$ and $d\in\cD$.
\end{df}

\begin{proposition}\label{equivalence_graded_simple}
Let $G$ and $H$ be groups. Let $\cD$ be a $G$-graded algebra and $\cD'$ an $H$-graded algebra. Suppose that $\cD$ and $\cD'$ are graded division algebras. Let $V$ be a graded right $\cD$-module and $V'$ a graded right $\cD'$-module. Let $\cR=\End_{\cD}(V)$ and $\cR'=\End_{\cD'}(V')$.
If $\psi\colon\cR\to\cR'$ is an equivalence of graded algebras, then there exists an equivalence $(\psi_0,\psi_1)$ from $(\cD,V)$ to $(\cD',V')$ such that $\psi_1(rv)=\psi(r)\psi_1(v)$ for all $r\in\cR$ and $v\in V$.
\end{proposition}

\begin{proof}
Let $I\subset\cR$ be a minimal graded left ideal. A standard argument shows that $I$ is generated by a homogeneous idempotent. Indeed, we have $I^2\neq 0$, because otherwise $I\cR$ would be a graded two-sided ideal with the property $(I\cR)^2=0$, which is impossible because $\cR$ is graded simple. Pick a homogeneous $x\in I$ such that $Ix\neq 0$. By the minimality of $I$, we have $Ix=I$ and ${\rm ann}_I(x)=0$, where ${\rm ann}_I(x)\bydef\{r\in I\;|\;rx=0\}$. Hence there exists $e\in I$ such that $ex=x$. Replacing $e$ by its identity component, we may assume that $e$ is homogeneous. Since $e^2-e\in{\rm ann}_I(x)$, we conclude that $e^2=e$. Since $\cR e\neq 0$, we have $\cR e=I$ by minimality.

Since $V$ is graded simple as a left $\cR$-module and $IV$ is a graded submodule of $V$, we have either $IV=0$ or $IV=V$. But the action of $\cR$ on $V$ is faithful, so $IV=V$. Pick a homogeneous $v\in V$ such that $Iv\neq 0$ and let $g=\deg v$. Then the map $I\to V$ given by $r\mapsto rv$ is a homomorphism of $\cR$-modules. By graded simplicity of $I$ and $V$, this map is an isomorphism. It sends $I_a$ to $V_{ag}$, $a\in G$, so it is an isomorphism of graded $\cR$-modules when regarded as a map $I^{[g]}\to V$. Now, $\End_\cR(I)$ can be identified with $e\cR e$ (as a graded algebra) in the usual way. Indeed, the right multiplication by a homogeneous $x\in e\cR e$ gives an endomorphism of $I$, which has the same degree as $x$. Conversely, any endomorphism of $I$ is easily seen to coincide with the right multiplication by the image of $e$. It follows that we have an isomorphism of graded algebras $e\cR e\to\End_\cR(V^{[g^{-1}]})$ sending $x\in e\cR e$ to the endomorphism $rv\mapsto rxv$, $r\in I$. 
Since $\cD=\End_{\cR}(V)$, we have obtained an equivalence $(e\cR e,I)\to(\cD,V)$, where $I\to V$ is an isomorphism of $\cR$-modules.

Since $I'=\psi(I)$ is a minimal graded left ideal of $\cR'$, which is generated by the homogeneous idempotent $e'=\psi(e)$, 
we also have an equivalence $(e'\cR' e',I')\to(\cD',V')$. Finally, restricting $\psi$ yields equivalences $\wt{\psi}_0\colon e\cR e\to e'\cR' e'$ and $\wt{\psi}_1\colon I\to I'$ such that $\psi_1(rx)=\psi(r)\psi_1(x)$ for $r\in\cR$ and $x\in I$. The compositions $\psi_0\colon\cD\to e\cR e\stackrel{\wt{\psi}_0}{\to} e'\cR' e'\to\cD'$ and $\psi_1\colon V\to I\stackrel{\wt{\psi}_1}{\to} I'\to V'$ have the desired properties.
\end{proof}

The map $\psi$ corresponding to $(\psi_0,\psi_1)$ in Proposition \ref{equivalence_graded_simple} can be expressed in the language of matrices as follows. Let  $\{v_1,\ldots,v_k\}$ be a homogeneous $\cD$-basis in $V$. Then any $r\in\cR$ is represented by a matrix $X=(x_{ij})\in M_k(\cD)$ relative to this $\cD$-basis. Also,  $\{\psi_1(v_1),\ldots,\psi_1(v_n)\}$ is a homogeneous $\cD'$-basis in $V'$, and $\psi$ is given by
\begin{equation}\label{isomorphism_in_matrix_form}
M_k(\cD)\to M_n(\cD')\colon (x_{ij})\mapsto (\psi_0(x_{ij})).
\end{equation}

From Theorems \ref{class_division} and \ref{class_matrix}, it is easy to obtain all fine gradings on $M_n(\FF)$. They were described in \cite{HPP98} over the field of complex numbers.

\begin{df}
Let $\cD=M_\ell(\FF)$ with a division grading \eqref{division_grading}. Let $k\geq 1$ be an integer. Let $\wt{G}$ be the direct product of $T$ and the free abelian group generated by the symbols $\wt{g}_1,\ldots,\wt{g}_k$.  Let $\cM(\cD,k)\bydef\cM(\wt{G},\cD,\kappa,\wt{\gamma})$ where $\kappa=(1,\ldots,1)$ and $\wt{\gamma}=(\wt{g}_1,\ldots,\wt{g}_k)$. Then $\cM(\cD,k)$ is isomorphic to the matrix algebra $M_n(\FF)$ where $n=k\ell$. We will denote the $\wt{G}$-grading on $M_n(\FF)$ arising from this isomorphism by $\M(\cD,k)$ or, abusing notation, by $\M(T,k)$, since the equivalence class of $\cD$ is uniquely determined by $T$.
\end{df}

$\M(\cD,k)$ can be described explicitly as follows. Pick a $\cD$-basis $\{v_1,\ldots,v_k\}$ in $V$ with $\deg(v_i)=\wt{g}_i$, $i=1,\ldots,k$. Then any $r\in\cR$ is represented by a matrix $X=(x_{ij})\in M_k(\cD)$ relative to this $\cD$-basis. Identify $M_k(\cD)$ with $M_k(\FF)\ot\cD$ via Kronecker product. Let $E_{ij}$ be the matrix units in $M_k(\FF)$. Then $\M(\cD,k)$ is given by
\[
\deg(E_{ij}\ot d)=\wt{g}_i(\deg d)\wt{g}_j^{-1}\quad\mbox{for all homogeneous}\quad d\in\cD,\;i,j=1,\ldots,k.
\]
This grading is fine in the class of abelian group gradings. The support consists of the elements $\wt{g}_i t\wt{g}_j^{-1}$, $t\in T$. The subgroup $\wt{G}^0$ that they generate is isomorphic to $\ZZ^{k-1}\times T$. It is the universal group of $\M(\cD,k)$.

\begin{corollary}\label{matr_fine}
Let $\Gamma$ be a fine abelian group grading on the matrix algebra $\cR=M_n(\FF)$ over an algebraically closed field $\FF$. Then $\Gamma$ is equivalent to some $\M(T,k)$ where $T\cong\ZZ_{\ell_1}^2\times\cdots\times\ZZ_{\ell_r}^2$ and  $k\ell_1\cdots\ell_r=n$. Two gradings $\M(T_1,k_1)$ and $\M(T_2,k_2)$ are equivalent if and only if $T_1\cong T_2$ and  $k_1=k_2$.\hfill{$\square$}
\end{corollary}

\subsection{The Weyl group of $\M(T,k)$}

Let $\cR=\cM(\cD,k)$ and let $V$ be the associated graded module, with a $\cD$-basis $\{v_1,\ldots,v_k\}$ where $\deg(v_i)=\wt{g}_i$, $i=1,\ldots,k$. Note that the homogeneous components of $V$ have the form $\FF v_i d$ where $d$ is a nonzero homogeneous element of $\cD$. Applying Proposition \ref{equivalence_graded_simple} to an equivalence $\psi\colon\cR\to\cR$, we see that there exists an equivalence $(\psi_0,\psi_1)$ of $(\cD,V)$ to itself such that $\psi_1(rv)=\psi(r)\psi_1(v)$ for all $r\in\cR$ and $v\in V$. Write $\psi_1(v_j)=\sum_i v_i d_{ij}$ for some $d_{ij}\in\cD$ and set $\Psi\bydef(d_{ij})$. Then \eqref{isomorphism_in_matrix_form} implies that $\psi\colon\cR\to\cR$ is given by
\begin{equation}\label{automorphism_in_matrix_form}
X\mapsto\Psi\psi_0(X)\Psi^{-1}
\end{equation}
where $\psi_0$ acts on $X$ entry-wise. Since $\psi_1$ is an equivalence of the graded space $V$ to itself, it must send $v_i$ to some $v_j d$ where $d$ is a nonzero homogeneous element of $\cD$. Hence there exists a permutation $\pi\in\sg(k)$ and nonzero homogeneous $d_1,\ldots,d_k$ such that $\psi_1(v_i)=v_{\pi(i)}d_i$. In other words,  $\Psi$ is the monomial matrix $PD$ where $P$ is the permutation matrix corresponding to $\pi$ (i.e., the matrix having $1$ in the $(i,\pi^{-1}(i))$-th positions and zeros elsewhere) and $D$ is the diagonal matrix $\diag(d_1,\ldots,d_k)$.

We see that everything boils down to the equivalence $\psi_0$ of $\cD$ to itself. The following result describes all such $\psi_0$.

\begin{proposition}\label{aut_division}
Let $\cD$ be a matrix algebra with a division grading $\Gamma_0$ given by \eqref{division_grading}. Let $T$ be the support and let $\beta$ be the nondegenerate alternating bicharacter $T\times T\to\FF^\times$ determined by \eqref{commutation_relations}. Then the mapping that sends $t\in T$ to the inner automorphism $X\mapsto X_t X X_t^{-1}$ is an isomorphism between $T$ and $\Stab(\Gamma_0)$. The quotient group $\Aut(\Gamma_0)/\Stab(\Gamma_0)$ is isomorphic to $\Aut(T,\beta)$.
\end{proposition}

\begin{proof}
The first assertion is proved in \cite{E10} and \cite{BK10}, but we include a proof for completeness. Any $\psi\in\Stab{\Gamma_0}$ must send $X_t$ to its scalar multiple, so we have a map $\lambda\colon T\to\FF^\times$ such that $\psi(X_t)=\lambda(t)X_t$ for all $t\in T$. One immediately verifies that $\lambda$ must be a group homomorphism, i.e., a character of $T$. It follows from  \eqref{commutation_relations} that the inner automorphism $X\mapsto X_tXX_t^{-1}$ corresponds to the character $\lambda(u)=\beta(t,u)$, $u\in T$. Since $\beta$ is nondegenerate, it establishes an isomorphism between $T$ and the group of characters $\wh{T}$.

Any $\psi\in\Aut{\Gamma_0}$ must send $X_t$ to a scalar multiple of some $X_u$. Hence we have a homomorphism $f\colon\Aut(\Gamma_0)\to\sg(T)$ such that $\psi(\FF X_t)=\FF X_{f(\psi)(t)}$ for all $t\in T$. Clearly, the kernel of $f$ is $\Stab(\Gamma_0)$. Fix $\psi\in\Aut(\Gamma_0)$ and let $\pi=f(\psi)$. Since $\psi(X_u X_v)\in\FF\psi(X_{uv})$ and $\psi(X_u)\psi(X_v)\in\FF X_{\pi(u)\pi(v)}$, we conclude that $\pi(uv)=\pi(u)\pi(v)$ for all $u,v\in T$, i.e., $\pi$ is an automorphism of $T$. Applying $\psi$ to both sides of \eqref{commutation_relations}, we obtain $\psi(X_u)\psi(X_v)=\beta(u,v)\psi(X_v)\psi(X_u)$. It follows that $\beta(\pi(u),\pi(v))=\beta(u,v)$, for all $u,v\in T$. We have proved that the image of $f$ is contained in $\Aut(T,\beta)$. Conversely, suppose we have $\pi\in\Aut(T,\beta)$. Observe that the algebra $\cD$ is generated by the elements $X_{a_i}$ and $X_{b_i}$, $i=1,\ldots,r$. They satisfy $X_{a_i}^{\ell_i}=1$, $X_{b_i}^{\ell_i}=1$, and the commutation relations given by \eqref{commutation_relations}. Clearly, the free algebra modulo these relations has dimension $\leq\ell_1\cdots\ell_r$. It follows that these relations are defining for the algebra $\cD$. Since $\pi(a_i)$ has the same order as $a_i$, we can choose a scalar multiple $X'_{\pi(a_i)}$ of $X_{\pi(a_i)}$ such that $(X'_{\pi(a_i)})^{\ell_i}=1$. Choose $X'_{\pi(b_i)}$ similarly. Since $\pi$ preserves $\beta$, the elements $X'_{\pi(a_i)}$ and $X'_{\pi(b_i)}$ will satisfy the same commutation relations as $X_{a_i}$ and $X_{b_i}$. It follows that there exists an automorphism $\psi$ of $\cD$ sending $X_{a_i}$ to $X'_{\pi(a_i)}$ and $X_{b_i}$ to $X'_{\pi(b_i)}$, for all $i=1,\ldots,r$. By construction, $f(\psi)$ coincides with $\pi$ on the elements $a_i$ and $b_i$, $i=1,\ldots,r$. Hence $f(\psi)=\pi$.
\end{proof}

Thus $\W(\Gamma_0)=\Aut(T,\beta)$. We are now ready to compute the Weyl group of any fine grading $\Gamma$ on $M_n(\FF)$.

\begin{theorem}\label{Weyl_fine_grad_matrix}
Let $\Gamma=\M(\cD,k)$ where $\cD$ is a matrix algebra with a division grading $\Gamma_0$ given by \eqref{division_grading}. Let $T$ be the support of $\Gamma_0$ and let $\beta$ be the nondegenerate alternating bicharacter $T\times T\to\FF^\times$ determined by \eqref{commutation_relations}. Then
\[
\W(\Gamma)\cong T^{k-1}\rtimes(\Aut(T,\beta)\times\sg(k))
\]
where $\Aut(T,\beta)$ and $\sg(k)$ act on $T^{k-1}$ through identifying the latter with $T^k/T$, where $T$ is imbedded into $T^k$ diagonally.
\end{theorem}

\begin{proof}
Recall that any $\psi\in\Aut(\Gamma)$ is given by \eqref{automorphism_in_matrix_form} where $\Psi=PD$, $P$ is a permutation matrix, $D=\diag(d_1,\ldots,d_k)$, and $\psi_0\in\Aut(\Gamma_0)$. Hence 
\begin{equation}\label{psi_Eij_d}
\psi(E_{ij}\ot d)=E_{\pi(i)\pi(j)}\ot d_i\psi_0(d)d_j^{-1}.
\end{equation}

Assume for a moment that  $\psi\in\Stab(\Gamma)$. Since all $\wt{g}_i\wt{g}_j^{-1}$, $i\neq j$, are distinct modulo $T$, we see by substituting $d=1$ in \eqref{psi_Eij_d} that $P=I$ and hence $D=\diag(\lambda_1\wh{d},\ldots,\lambda_k\wh{d})$ for some $\lambda_i\in\FF^\times$ and nonzero homogeneous $\wh{d}\in\cD$. We may simultaneously replace $D$ by $\wh{D}=\diag(\lambda_1,\ldots,\lambda_k)$ and $\psi_0$ by $\wh{\psi}_0$ where $\wh{\psi}_0(d)=\wh{d}\psi_0(d)\wh{d}^{-1}$. Then taking $i=j$ in \eqref{psi_Eij_d}, we see that $\psi_0\in\Stab(\Gamma_0)$. Conversely, for $P=I$ and any $D=\diag(\lambda_1,\ldots,\lambda_k)$ and $\psi_0\in\Stab(\Gamma_0)$, equation \eqref{automorphism_in_matrix_form} yields $\psi\in\Stab(\Gamma)$. We have proved that $\Stab(\Gamma)$ is isomorphic to $\big((\FF^\times)^k/\FF^\times\big)\times\Stab(\Gamma_0)$. But by Proposition \ref{aut_division}, we know that $\Stab(\Gamma_0)$ is isomorphic to $T$ via $t\mapsto\Ad(X_t)$ where, for invertible $d\in\cD$, $\Ad(d)\colon\cD\to\cD$ is the inner automorphism $x\mapsto dxd^{-1}$.

For any $\psi_0\in\Aut(\Gamma_0)$, $\pi\in\sg(k)$, and nonzero homogeneous $d_1,\ldots,d_s$, equation \eqref{automorphism_in_matrix_form} defines a map $\psi\colon\cR\to\cR$ where $\Psi=PD$, $P$ is the permutation matrix corresponding to $\pi$ and $D=\diag(d_1,\ldots,d_k)$. Then for nonzero homogeneous $d$, the element $\psi(E_{ij}\ot d)$ is also homogeneous. Since the degree of $\psi(E_{ii}\ot d)$ does not depend on $i$ and, for $i\neq j$, the element $E_{ij}\ot d$ spans its homogeneous component, we conclude that $\psi$ is an equivalence. It follows that the homomorphism $\Aut(\Gamma)\to\sg(k)$ is onto and can be split by sending a permutation $\pi\in\sg(k)$ to the equivalence $\psi\in\Aut(\Gamma)$ corresponding to $\psi_0=\id$ and $D=I$.
Let $\wt{T}$ be the multiplicative group of nonzero homogeneous elements of $\cD$, which is a central extension of $T$ by $\FF^\times$.
The kernel $K$ of the homomorphism $\Aut(\Gamma)\to\sg(k)$ consists of the equivalences that correspond to $P=I$, so we have an epimorphism
\[
f\colon(\wt{T}^k/\FF^\times)\rtimes\Aut(\Gamma_0)\to K.
\]
The kernel $K_0$ of $f$ consists of the elements of the form $((d,\ldots,d)\FF^\times,\Ad(d^{-1}))$, $d\in\wt{T}$.

Let $N=A\times\Stab(\Gamma_0)$ where $A\subset\wt{T}^k/\FF^\times$ consists of the elements of the form $(\lambda_1 d,\ldots,\lambda_k d)\FF^\times$ where $d\in\wt{T}$ and $\lambda_i\in\FF^\times$. Then $N$ is a normal subgroup of $(\wt{T}^k/\FF^\times)\rtimes\Aut(\Gamma_0)$. Clearly, $N\supset K_0$, and $f$ maps $N$ onto $\Stab(\Gamma)$. It follows that
\begin{align*}
K/\Stab(\Gamma)&\cong((\wt{T}^k/\FF^\times)\rtimes\Aut(\Gamma_0))/(A\times\Stab(\Gamma_0))\\
&\cong((\wt{T}^k/\FF^\times)/A)\rtimes(\Aut(\Gamma_0)/\Stab(\Gamma_0)).
\end{align*}
Now $(\wt{T}^k/\FF^\times)/A\cong T^k/T$ where $T$ is imbedded into $T^k$ diagonally. Also, we have  $\Aut(\Gamma_0)/\Stab(\Gamma_0)\cong\Aut(T,\beta)$ by Proposition \ref{aut_division}.
Hence
\[
\Aut(\Gamma)/\Stab(\Gamma)\cong ((T^k/T)\rtimes\Aut(T,\beta))\rtimes\sg(k).
\]
It remains to observe that the actions of $\Aut(T,\beta)$ and of $\sg(k)$ on $T^k/T$ commute with each other.
\end{proof}

Recall that $U(\Gamma)$ is isomorphic to $\ZZ^{k-1}\times T$. To describe the action of $\W(\Gamma)$ on $U(\Gamma)$, it is convenient to realize them as follows:
\begin{align}
\W(\Gamma)&=(T^k/T)\rtimes(\Aut(T,\beta)\times\sg(k)),\label{Weyl_group_matrix} \\
U(\Gamma)&=\ZZ^k_0\times T,\label{univ_group_matrix}
\end{align}
where $\ZZ^k_0$ is the subgroup of $\ZZ^k$ consisting of all $\ul{x}=(x_1,\ldots,x_k)$ such that $\sum_i x_i=0$ (in other words, the elements $\wt{g}_i$ are identified with the standard basis in $\ZZ^k$).

\begin{corollary}[of the proof]
Writing $\W(\Gamma)$ as in \eqref{Weyl_group_matrix} and $U(\Gamma)$ as in \eqref{univ_group_matrix}, the action of $\W(\Gamma)$ on $U(\Gamma)$ is the following:
\begin{itemize}
\item $(t_1,\ldots,t_k)\in T^k$ acts by sending $(\ul{x},t)$ to $(\ul{x},t\prod_i t_i^{x_i})$, and this action of $T^k$ factors through $T^k/T$;
\item $\alpha\in\Aut(T,\beta)$ acts by sending $(\ul{x},t)$ to $(\ul{x},\alpha(t))$;
\item $\pi\in\sg(k)$ acts on $(\ul{x},t)$ by permuting the components of $\ul{x}$.\hfill{$\square$}
\end{itemize}
\end{corollary}

\begin{remark}
We have also proved that $\Stab(\Gamma)=\Diag(\Gamma)\cong(\FF^\times)^{k-1}\times T$. This is a special case of \cite[Proposition 2.8]{BK10}, where $\Stab(\Gamma)$ is computed for any abelian group grading $\Gamma$ on $M_n(\FF)$.
\end{remark}

\begin{remark}
The group $\Aut(T,\beta)$ can be explicitly computed as follows. First, decompose $T$ into primary components: $T=\prod_i T_i$ where $T_i$ is a $q_i$-group, $q_i$ is a prime, $q_i\neq\chr{\FF}$. Then $T_i$ are $\beta$-orthogonal, so we have
\[
\Aut(T,\beta)=\prod_i\Aut(T_i,\beta_i)\quad\mbox{where}\quad\beta_i=\beta|_{T_i\times T_i}.
\]
So it is sufficient to consider the case when $T$ is a $q$-group. Then \eqref{decomp_T} yields
\[
T\cong\big((\ZZ/{q^{\alpha_1}}\ZZ)\times(\ZZ/{q^{\alpha_1}}\ZZ)\big)^{m_1}\times\cdots\times\big((\ZZ/{q^{\alpha_f}}\ZZ)\times(\ZZ/{q^{\alpha_f}}\ZZ)\big)^{m_f}
\]
where $\alpha_1<\ldots<\alpha_f$ and $m_i>0$. Let $m=m_1+\cdots+m_f$. Then $\Aut(T,\beta)$ can be identified with the group of $(2m\times 2m)$-matrices $A$ of the following form: $A$ is partitioned into blocks $A_{ij}$ of sizes $2m_i\times 2m_j$, $i,j=1,\ldots,f$, the entries of $A_{ij}$ are integers modulo $q^{\alpha_i}$, $A_{ij}\equiv 0\pmod{q^{\alpha_i-\alpha_j}}$ for all $i>j$, and ${}^tA J A\equiv J\pmod{q^{\alpha_f}}$ where ${}^tA$ is the transpose of $A$ and $J$ is the diagonal sum of blocks $q^{\alpha_f-\alpha_i}\begin{pmatrix}0&1\\-1&0\end{pmatrix}^{\oplus m_i}$, $i=1,\ldots,f$.
\end{remark}


\section{The algebra of octonions}\label{se:octonions}

In this section, the Weyl groups of the fine gradings on the algebra of octonions will be computed.
The \emph{Cayley algebra} $\cC$, or \emph{algebra of octonions}, over $\bF$ is the unique, up to isomorphism, eight-dimensional unital composition algebra (recall that we are assuming $\FF$ algebraically closed). There exists a nondegenerate quadratic form (the {\em norm}) $n\colon  \cC\rightarrow \bF$ such that $n(xy)=n(x)n(y)$ for all  $x,y\in\cC$. Here the norm being nondegenerate means that its polar form: $n(x,y)=n(x+y)-n(x)-n(y)$ is a nondegenerate symmetric bilinear form. Note that $n(x)=\frac12 n(x,x)$.

The next result summarizes some of the well-known properties of this algebra (see \cite[Chapter VIII]{KMRT} and \cite[Chapter 2]{ZSSS}):

\begin{proposition}\label{pr:Cayley}
Let $\cC$ be the Cayley algebra over $\bF$. Then:
\begin{enumerate}
\item[1)] Any $x\in\cC$ satisfies the degree $2$ Cayley--Hamilton equation:
\[
x^2-n(x,1)x+n(x)1=0.
\]

\item[2)] The map $x\mapsto \bar x=n(x,1)1-x$ is an involution, called the \emph{standard conjugation}, of $\cC$ and for any $x,y,z\in\cC$, we have $x\bar x=\bar xx=n(x)1$ and $n(xy,z)=n(y,\bar xz)=n(x,z\bar y)$.

\item[3)] There is a \emph{``good basis''} $\{e_1,e_2,u_1,u_2,u_3,v_1,v_2,v_3\}$ of $\cC$ consisting of iso\-tro\-pic elements, such that $n(e_1,e_2)=n(u_i,v_i)=1$ for any $i=1,2,3$ and $n(e_r,u_i)=n(e_r,v_i)=n(u_i,u_j)=n(u_i,v_j)=n(v_i,v_j)=0$ for any $r=1,2$ and $1\leq i\ne j\leq 3$, whose multiplication table is shown in Figure \ref{fig:Cayley}.
\begin{figure}
\[ 
\vbox{\offinterlineskip
\halign{\hfil$#$\enspace\hfil&#\vreglon
 &\hfil\enspace$#$\enspace\hfil
 &\hfil\enspace$#$\enspace\hfil&#\vregleta
 &\hfil\enspace$#$\enspace\hfil
 &\hfil\enspace$#$\enspace\hfil
 &\hfil\enspace$#$\enspace\hfil&#\vregleta
 &\hfil\enspace$#$\enspace\hfil
 &\hfil\enspace$#$\enspace\hfil
 &\hfil\enspace$#$\enspace\hfil&#\vreglon\cr
 &\omit\hfil\vrule width 1pt depth 4pt height 10pt
   &e_1&e_2&\omit&u_1&u_2&u_3&\omit&v_1&v_2&v_3&\omit\cr
 \noalign{\hreglon}
 e_1&&e_1&0&&u_1&u_2&u_3&&0&0&0&\cr
 e_2&&0&e_2&&0&0&0&&v_1&v_2&v_3&\cr
 &\multispan{11}{\hregletafill}\cr
 u_1&&0&u_1&&0&v_3&-v_2&&-e_1&0&0&\cr
 u_2&&0&u_2&&-v_3&0&v_1&&0&-e_1&0&\cr
 u_3&&0&u_3&&v_2&-v_1&0&&0&0&-e_1&\cr
 &\multispan{11}{\hregletafill}\cr
 v_1&&v_1&0&&-e_2&0&0&&0&u_3&-u_2&\cr
 v_2&&v_2&0&&0&-e_2&0&&-u_3&0&u_1&\cr
 v_3&&v_3&0&&0&0&-e_2&&u_2&-u_1&0&\cr
 &\multispan{12}{\hreglonfill}\cr}}
\]
\caption{Multiplication table of the Cayley algebra}\label{fig:Cayley}
\end{figure}
\end{enumerate}
\end{proposition}

The linear map $t(x)=n(x,1)$ is called the {\em trace}. A crucial step in the classification of fine gradings on $\cC$ \cite{E98} is the fact that, for any grading $\cC=\bigoplus_{g\in G}\cC_g$, we have $t(\cC_g \cC_h)=0$ unless $gh=e$.

\subsection{Fine gradings on the algebra of octonions}

A ``good basis'' of the Cayley algebra $\cC$ gives a $\ZZ^2$-grading with
\[
\begin{array}{lr}
\cC_{(0,0)}=\bF e_1\oplus \bF e_2,& \\
\cC_{(1,0)}=\bF u_1,& \cC_{(-1,0)}=\bF v_1,\\
\cC_{(0,1)}=\bF u_2,& \cC_{(0,-1)}=\bF v_2,\\
\cC_{(1,1)}=\bF v_3,& \cC_{(-1,-1)}=\bF u_3.
\end{array}
\]
This is called the \emph{Cartan grading} of the Cayley algebra. It is fine, and $\bZ^2$ is its universal group.

\smallskip

Let $\cQ$ be a proper four-dimensional subalgebra of the Cayley algebra $\cC$ such that $n\vert_\cQ$ is nondegenerate, and let $u$ be any element in $\cC\setminus\cQ$ with $n(u)=\alpha\ne 0$. Then $\cC=\cQ\oplus\cQ u$ and we get:
\[
\begin{split}
&n(a+bu)=n(a)+\alpha n(b),\\
&(a+bu)(c+du)=(ac-\alpha\bar db) + (da+b\bar c)u,
\end{split}
\]
for any $a,b,c,d\in\cQ$. Then $\cC$ is said to be obtained from $\cQ$ by means of the \emph{Cayley--Dickson doubling process} and we write $\cC=\CD(\cQ,\alpha)$. This gives a $\ZZ_2$-grading on $\cC$ with $\cC\subo=\cQ$ and $\cC\subuno=\cQ u$.

The subalgebra $\cQ$ above is a quaternion subalgebra which in turn can be obtained from a two-dimensional subalgebra $\cK$ through the same process $\cQ=\CD(\cK,\beta)=\cK\oplus\cK v$, and this gives a $\bZ_2$-grading of $\cQ$ and hence a $\bZ_2^2$-grading of $\cC=\cK\oplus\cK v\oplus\cK u\oplus(\cK v)u$. We write here $\cQ=\CD(\cK,\beta,\alpha)$.

If $\chr{\FF}\ne 2$, then $\cK$ can be obtained in turn from the ground field: $\cK=\CD(\bF,\gamma)$, and a $\bZ_2^3$-grading on $\cC$ appears. Here we write $\cC=\CD(\bF,\gamma,\beta,\alpha)$.

These gradings by $\bZ_2^r$, $r=1,2,3$, will be called \emph{gradings induced by the Cayley--Dickson doubling process}. The groups $\bZ_2^r$ are their universal groups.
Since there is a unique $d$-dimensional unital composition algebra for each $d=2,4,8$, these $\bZ_2^r$-gradings, $r=1,2,3$, are unique up to equivalence.

\smallskip

The classification of fine gradings on $\cC$ was obtained in \cite{E98}:

\begin{theorem}\label{th:grOctonions}
Let $\Gamma$ be a fine abelian group grading on the Cayley algebra $\cC$ over an algebraically closed field $\FF$.  Then $\Gamma$ is equivalent either to the Cartan grading or to the $\bZ_2^3$-grading induced by the Cayley--Dickson doubling process. The latter grading does not occur if $\chr{\FF}=2$.\hfill{$\square$}
\end{theorem}

\subsection{Cartan grading}

Let $S$ be the vector subspace spanned by $(1,1,1)$ in $\bR^3$ and consider the two-dimensional real vector space $E=\bR^3/S$. Take the elements
\[
\epsilon_1=(1,0,0)+S,\ \epsilon_2=(0,1,0)+S,\ \epsilon_3=(0,0,1)+S.
\]
The subgroup $G=\bZ\epsilon_1+\bZ\epsilon_2+\bZ\epsilon_3$ is isomorphic to $\bZ^2$, and we may think of the Cartan grading $\Gamma$ of the Cayley algebra $\cC$ as the grading in which
\[
\begin{split}
\degree(e_1)&=0=\degree(e_2),\\
\degree(u_i)&=\epsilon_i=-\degree(v_i),\; i=1,2,3.
\end{split}
\]
Then $\Supp\Gamma=\{0\}\cup\{\pm\epsilon_i\;|\; i=1,2,3\}$. The set
\[
\Phi\bydef \Bigl(\Supp\Gamma\cup\{\alpha+\beta\;|\;\alpha,\beta\in\Supp\Gamma, \alpha\ne \pm\beta\}\Bigr)\setminus\{0\}
\]
is the root system of type $G_2$ (although presented in a slightly different way from \cite[Chapter VI.4.13]{Bourbaki}).

Identifying the Weyl group $\W(\Gamma)$ with a subgroup of $\Aut(G)$, and this with a subgroup of $GL(E)$, we have:
\[
\begin{split}
\W(\Gamma)&\subset \{\mu\in \Aut(G)\;|\; \mu(\Supp\Gamma)=\Supp\Gamma\}\\
          &\subset \{\mu\in GL(E)\;|\; \mu(\Phi)=\Phi\}\defby \Aut\Phi.
\end{split}
\]
The latter group is the automorphism group of the root system $\Phi$, which coincides with its Weyl group.

If $\chr{\FF}\ne 2,3$, then we can work with the Lie algebra $\Der(\cC)$ and prove the next theorem using known results on the simple Lie algebra of type $G_2$ in \cite{Seligman}. The proof below works directly with the Cartan grading on the Cayley algebra and is valid in any characteristic.

\begin{theorem}\label{th:WeylCartanCayley}
Let $\Gamma$ be the Cartan grading on the Cayley algebra over an algebraically closed field. Identify $\Supp\Gamma\setminus\{0\}$ with the short roots in the root system $\Phi$ of type $G_2$. Then $\W(\Gamma)=\Aut\Phi$.
\end{theorem}
\begin{proof}
The group $\Aut\Phi$ is the dihedral group of order $12$. Now the order $3$ automorphism of $\cC$:
\begin{equation}\label{eq:tauC}
\tau\colon e_j\mapsto e_j,\ u_i\mapsto u_{i+1},\ v_i\mapsto v_{i+1},
\end{equation}
for $j=1,2$ and $i=1,2,3$ (modulo $3$), belongs to $\Aut(\Gamma)$, and its projection into $\W(\Gamma)$ permutes cyclically the $\epsilon_i$'s. Also, the order two automorphisms:
\[
\begin{split}
\varphi_1&\colon e_1\leftrightarrow e_2,\ u_i\leftrightarrow v_i\ (i=1,2,3),\\
\varphi_2&\colon e_j\mapsto e_j\,(j=1,2),\ u_1\mapsto -u_1,\ u_2\leftrightarrow u_3,\ v_1\mapsto -v_1,\ v_2\leftrightarrow v_3,
\end{split}
\]
belong to $\Aut(\Gamma)$, and their projections into $\W(\Gamma)$ generate a subgroup of order $4$. Therefore, the order of $\W(\Gamma)$ is at least $4\times 3=12$, and hence $\W(\Gamma)$ is the whole $\Aut\Phi$.
\end{proof}

\begin{remark}
We have $\Stab(\Gamma)=\Diag(\Gamma)$. It is a maximal torus in the algebraic group $\Aut(\cC)$.
\end{remark}

\subsection{$\ZZ_2^3$-grading}

Recall that this grading occurs only if $\chr{\FF}\ne 2$. We fix the following notation: let
\[
c_1=(\bar 1,\bar 0,\bar 0),\ c_2=(\bar 0,\bar 1,\bar 0),\ c_3=(\bar 0,\bar 0,\bar 1)
\]
be the standard basis of $\ZZ_2^3$. The Cayley algebra $\cC$ is obtained by repeated application of the Cayley--Dickson doubling process:
\[
\calK=\bF\oplus\bF w_1,\quad \calQ=\calK\oplus\calK w_2,\quad \calC=\calQ\oplus\calQ w_3,
\]
with $w_i^2=1$ for $i=1,2,3$. (One may take $w_1=e_1-e_2$, $w_2=u_1-v_1$ and $w_3=u_2-v_2$.) Setting
\begin{equation}\label{CD_grading}
\deg w_i=c_i,\; i=1,2,3,
\end{equation}
gives the $\bZ_2^3$-grading $\Gamma$ induced by the Cayley--Dickson doubling process. 

\begin{theorem}\label{th:Z23Weyl}
Let $\Gamma$ be the $\ZZ_2^3$-grading on the Cayley algebra as in \eqref{CD_grading} over an algebraically closed field of characteristic different from $2$. 
Then $\W(\Gamma)=\Aut(\bZ_2^3)\cong GL_3(2)$.
\end{theorem}

\begin{proof}
Given any $\mu\in\Aut(\bZ_2^3)$, take $\wt{w}_i\in\cC_{\mu(c_i)}$ with $\wt{w}_i^2=1$. Then $\cC$ is obtained by repeated application of the Cayley--Dickson doubling process:
\[
\wt{\calK}=\bF\oplus\bF \wt{w}_1,\quad \wt{\calQ}=\wt{\calK}\oplus\wt{\calK} \wt{w}_2,\quad \calC=\wt{\calQ}\oplus\wt{\calQ} \wt{w}_3,
\]
and hence there is a unique automorphism $\varphi\in\Aut\cC$ such that $\varphi(w_i)=\wt{w}_i$ for $i=1,2,3$. Then $\varphi$ belongs to $\Aut(\Gamma)$, and its projection into $\W(\Gamma)$ is precisely $\mu$. This shows that $\W(\Gamma)$ fills the whole $\Aut(\bZ_2^3)$.
\end{proof}

\begin{remark} As any $\varphi\in\Stab(\Gamma)$ multiplies each $w_i$, $i=1,2,3$, by either $1$ or $-1$, we see that $\Stab(\Gamma)=\Diag(\Gamma)$ is isomorphic to $\bZ_2^3$. Therefore, the group $\Aut(\Gamma)$ is a (non-split) extension of $\bZ_2^3$ by $\W(\Gamma)\cong GL_3(2)$. This group $\Aut(\Gamma)$, in its irreducible seven-dimensional representation given by the trace zero elements in $\calC$, is precisely the group used in \cite{Wilson} to give a nice construction of the compact real form of the Lie algebra of type $G_2$.
\end{remark}


\section{The Albert algebra}\label{se:Albert}

In this section, the Weyl groups of the fine gradings on the Albert algebra will be computed. We assume $\chr{\FF}\ne 2$ throughout this section.

Let $\calC$ be the Cayley algebra. The \emph{Albert algebra} is the algebra of Hermitian $3\times 3$-matrices over $\calC$:
\begin{equation*}
\begin{split}
\calA=\sym_3(\calC,*)&=\left\{\begin{pmatrix} \alpha_1&\bar a_3&a_2\\ a_3&\alpha_2&\bar a_1\\ \bar a_2&a_1&\alpha_3\end{pmatrix}\;|\; 
\alpha_1,\alpha_2,\alpha_3\in \bF,\; a_1,a_2,a_3\in \calC\right\} \\[6pt]
&=\bF E_1\oplus\bF E_2\oplus\bF E_3\oplus \iota_1(\calC)\oplus\iota_2(\calC)\oplus\iota_3(\calC),
\end{split}
\end{equation*}
where
\[
\begin{aligned}
E_1&=\begin{pmatrix}1&0&0\\ 0&0&0\\ 0&0&0\end{pmatrix}, &
E_2&=\begin{pmatrix}0&0&0\\ 0&1&0\\ 0&0&0\end{pmatrix}, &
E_3&=\begin{pmatrix}1&0&0\\ 0&0&0\\ 0&0&1\end{pmatrix}, \\
\iota_1(a)&=2\begin{pmatrix}0&0&0\\ 0&0&\bar a\\ 0&a&0\end{pmatrix},\quad &
\iota_2(a)&=2\begin{pmatrix}0&0&a\\ 0&0&0\\ \bar a&0&0\end{pmatrix},\quad &
\iota_3(a)&=2\begin{pmatrix}0&\bar a&0\\a&0&0\\ 0&0&0\end{pmatrix},\quad
\end{aligned}
\]
with (commutative) multiplication given by $X Y=\frac{1}{2}(X\cdot Y+Y\cdot X)$, where $X\cdot Y$ denotes the usual product of matrices $X$ and $Y$. Then $E_i$ are orthogonal idempotents with $E_1+E_2+E_3=1$. The rest of the products are as follows:
\begin{equation}\label{eq:Albertproduct}
\begin{split}
&E_i\iota_i(a)=0,\quad E_{i+1}\iota_i(a)=\frac{1}{2}\iota_i(a)=E_{i+2}\iota_i(a),\\
&\iota_i(a)\iota_{i+1}(b)=\iota_{i+2}(\bar a\bar b),\quad
\iota_i(a)\iota_i(b)=2n(a,b)(E_{i+1}+E_{i+2}),
\end{split}
\end{equation}
for any $a,b\in \calC$, with $i=1,2,3$ taken modulo $3$. (This convention about indices will be used without further mention.)

\smallskip

For the main properties of the Albert algebra the reader may consult \cite{Jacobson}. It is the only exceptional simple Jordan algebra over $\bF$.
Any $X\in\cA$ satisfies the degree $3$ Cayley--Hamilton equation:
\[
X^3-T(X)X^2+S(X)X-N(X)1=0,
\]
where the linear form $T$ is called the {\em trace} and the cubic form $N$ is called the {\em norm} of the Albert algebra. A crucial step in the classification of fine gradings on $\cA$ \cite{Albert} is the fact that, for any grading $\cA=\bigoplus_{g\in G}\cA_g$, we have $T(\cA_g \cA_h)=0$ unless $gh=e$. We also note that any automorphism $\vphi$ of $\cC$ extends to $\cA$ by setting $\varphi(E_i)=E_i$, $\varphi(\iota_i(x))=\iota_i(\varphi(x))$, for all $x\in\cC$ and $i=1,2,3$.

\subsection{Fine gradings on the Albert algebra}

First we describe the gradings in question as they are presented in \cite{Albert}.

\smallskip

\noindent\ul{Cartan grading}:\quad
Consider the following elements in $\ZZ^4=\ZZ^2\times\ZZ^2$:
\[
\begin{aligned}
a_1&=(1,0,0,0),\quad& a_2&=(0,1,0,0),\quad&a_3&=(-1,-1,0,0),\\
g_1&=(0,0,1,0),&g_2&=(0,0,0,1),&g_3&=(0,0,-1,-1).
\end{aligned}
\]
Then $a_1+a_2+a_3=0=g_1+g_2+g_3$. Take a ``good basis'' of the Cayley algebra $\calC$. Recall that the assignment
\[
\degree e_1=\degree e_2=0,\quad \degree u_i=g_i=-\degree v_i
\]
gives the Cartan grading on $\calC$.

Now, the assignment
\[
\begin{split}
&\degree E_i=0,\\
&\degree\iota_i(e_1)=a_i=-\degree\iota_i(e_2),\\
&\degree\iota_i(u_i)=g_i=-\degree\iota_i(v_i),\\
&\degree\iota_i(u_{i+1})=a_{i+2}+g_{i+1}=-\degree\iota_i(v_{i+1}),\\
&\degree\iota_i(u_{i+2})=-a_{i+1}+g_{i+2}=-\degree\iota_i(v_{i+2}),
\end{split}
\]
where $i=1,2,3$, gives a $\ZZ^4$-grading on the Albert algebra $\calA$. (To see this, it suffices to look at the first component of $\ZZ^2\times\ZZ^2$, and by the cyclic symmetry of the product, it is enough to check that $\degree\bigl(\iota_3(\bar x\bar y)\bigr)=\degree\iota_1(x)+\degree\iota_2(y)$ for any $x,y$ in the ``good basis'' of $\cC$, and this is straightforward.)

This $\ZZ^4$-grading will be called the \emph{Cartan grading} on the Albert algebra. It is fine, and $\bZ^4$ is its universal group (see \cite{Albert}).

\smallskip

\noindent\ul{$\bZ_2^5$-grading}:\quad
Recall the $\ZZ_2^3$-grading \eqref{CD_grading} on the Cayley algebra $\calC$ induced by the Cayley--Dickson doubling process.
Then $\calA$ is obviously $\bZ_2^5$-graded as follows:
\[
\begin{split}
\degree E_i&=(\bar 0,\bar 0,\bar 0,\bar 0,\bar 0),\ i=1,2,3,\\
\degree\iota_1(x)&=(\bar 1,\bar 0,\degree x),\\
\degree\iota_2(x)&=(\bar 0,\bar 1,\degree x),\\
\degree\iota_3(x)&=(\bar 1,\bar 1,\degree x),
\end{split}
\]
for homogeneous elements $x\in\calC$. 

This grading will be referred to as the {\em $\bZ_2^5$-grading} on the Albert algebra. It is fine, and $\bZ_2^5$ is its universal group (see \cite{Albert}).

\smallskip

\noindent\ul{$\bZ\times\bZ_2^3$-grading}:\quad
Take an element $\bi\in\bF$ with $\bi^2=-1$ and consider the following elements in $\calA$:
\[
\begin{split}
E&=E_1,\\
\wt{E}&=1-E=E_2+E_3,\\
\nu(a)&=\bi\iota_1(a)\quad\mbox{for all}\quad a\in\calC_0,\\
\nu_{\pm}(x)&=\iota_2(x)\pm \bi\iota_3(\bar x)\quad\mbox{for all}\quad x\in \calC,\\
S^{\pm}&=E_3-E_2\pm\frac{\bi}{2}\iota_1(1),
\end{split}
\]
where $\cC_0=\{a\in\cC\;|\;t(a)=0\}$. The above elements span $\calA$, and \eqref{eq:Albertproduct} translates to:
\[
\begin{split}
&E\wt{E}=0,\quad ES^{\pm}=0,\quad E\nu(a)=0,\quad E\nu_{\pm}(x)=\frac{1}{2}\nu_{\pm}(x),\\
&\wt{E} S^{\pm}=S^{\pm},\quad \wt{E}\nu(a)=\nu(a),\quad \wt{E}\nu_{\pm}(x)=\frac{1}{2}\nu_{\pm}(x),\\
&S^+S^-=2\wt{E},\quad S^{\pm}\nu(a)=0,\quad S^{\pm}\nu_{\mp}(x)=\nu_{\pm}(x),\quad S^{\pm}\nu_{\pm}(x)=0,\\
&\nu(a)\nu(b)=-2n(a,b)\wt{E},\quad \nu(a)\nu_{\pm}(x)=\pm\nu_{\pm}(xa),\\
&\nu_{\pm}(x)\nu_{\pm}(y)=2n(x,y)S^{\pm},\quad
\nu_+(x)\nu_-(y)=2n(x,y)(2E+\wt{E})+\nu(\bar xy-\bar yx),
\end{split}
\]
for all $x,y\in\calC$ and $a,b\in \calC_0$.

There appears a $\ZZ$-grading on $\calA$:
\begin{equation*}
\calA=\calA_{-2}\oplus\calA_{-1}\oplus\calA_0\oplus\calA_1\oplus\calA_2,
\end{equation*}
with $\calA_{\pm 2}=\bF S^{\pm}$, $\calA_{\pm 1}=\nu_{\pm}(\calC)$, and $\calA_0=\bF E\oplus\Bigl(\bF \wt{E}\oplus \nu(\calC_0)\Bigr)$.
The $\bZ_2^3$-grading on $\calC$ combines with this $\bZ$-grading to give a $\bZ\times\bZ_2^3$-grading as follows:
\begin{equation}\label{eq:Z_Z23_grading}
\begin{split}
\degree S^\pm&=(\pm 2,\bar 0,\bar 0,\bar 0),\\
\degree\nu_{\pm}(x)&=(\pm 1,\degree x),\\
\degree E&=0=\degree\wt{E},\\
\degree\nu(a)&=(0,\degree a),
\end{split}
\end{equation}
for homogeneous elements $x\in\calC$ and $a\in\calC_0$.

This grading will be referred to as the {\em $\bZ\times\bZ_2^3$-grading} on the Albert algebra. It is fine, and $\bZ\times\bZ_2^3$ is its universal group (see \cite{Albert}).

\smallskip

\noindent\ul{$\bZ_3^3$-grading}:\quad
Consider the order $3$ automorphism $\tau$ of $\calC$ in \eqref{eq:tauC}, and the new multiplication defined on $\calC$ by:
\[
x*y\bydef\tau(\bar x)\tau^2(\bar y)\quad\mbox{for all}\quad x,y\in \calC.
\]
Then $n(x*y)=n(x)n(y)$ for all $x,y$, since $\tau$ preserves the norm. Moreover, for all $x,y,z\in \calC$, we have:
\[
\begin{split}
n(x*y,z)&=n(\tau(\bar x)\tau^2(\bar y),z)\\
    &=n(\tau(\bar x),z\tau^2(y))\\
    &=n(\bar x,\tau^2(z)\tau(y))\\
    &=n(x,\tau(\bar y)\tau^2(\bar z))\\
    &=n(x,y*z).
\end{split}
\]
Hence $(\calC,*,n)$ is a symmetric composition algebra (see \cite{ElduqueGrSym} or \cite[Chapter VIII]{KMRT}). Actually, $(\calC,*)$ is the Okubo algebra over $\bF$. Its multiplication table is shown in Figure \ref{fig:Okubo}.
\begin{figure}
\[
\vbox{\offinterlineskip
\halign{\hfil$#$\enspace\hfil&#\vreglon
 &\hfil\enspace$#$\enspace\hfil
 &\hfil\enspace$#$\enspace\hfil&#\vregleta
 &\hfil\enspace$#$\enspace\hfil
 &\hfil\enspace$#$\enspace\hfil&#\vregleta
 &\hfil\enspace$#$\enspace\hfil
 &\hfil\enspace$#$\enspace\hfil&#\vregleta
 &\hfil\enspace$#$\enspace\hfil
 &\hfil\enspace$#$\enspace\hfil&#\vreglon\cr
 &\omit\hfil\vrule width 1pt depth 4pt height 10pt
   &e_1&e_2&\omit&u_1&v_1&\omit&u_2&v_2&\omit&u_3&v_3&\omit\cr
 \noalign{\hreglon}
 e_1&&e_2&0&&0&-v_3&&0&-v_1&&0&-v_2&\cr
 e_2&&0&e_1&&-u_3&0&&-u_1&0&&-u_2&0&\cr
 &\multispan{12}{\hregletafill}\cr
 u_1&&-u_2&0&&v_1&0&&-v_3&0&&0&-e_1&\cr
 v_1&&0&-v_2&&0&u_1&&0&-u_3&&-e_2&0&\cr
 &\multispan{12}{\hregletafill}\cr
 u_2&&-u_3&0&&0&-e_1&&v_2&0&&-v_1&0&\cr
 v_2&&0&-v_3&&-e_2&0&&0&u_2&&0&-u_1&\cr
 &\multispan{12}{\hregletafill}\cr
 u_3&&-u_1&0&&-v_2&0&&0&-e_1&&v_3&0&\cr
 v_3&&0&-v_1&&0&-u_2&&-e_2&0&&0&u_3&\cr
 &\multispan{13}{\hreglonfill}\cr}}
\]
\caption{Multiplication table of the Okubo algebra}\label{fig:Okubo}
\end{figure}

This Okubo algebra is $\bZ_3^2$-graded by setting $\degree e_1=(\bar 1,\bar 0)$ and 
$\degree u_1=(\bar 0,\bar 1)$, 
with the degrees of the remaining elements being uniquely determined.

Assume now that $\chr{\FF}\ne 3$. Then this $\bZ_3^2$-grading is determined by two commuting order $3$ automorphisms $\varphi_1,\varphi_2\in\Aut(\calC,*)$:
\[
\varphi_1(e_1)=\omega e_1,\,\varphi_1(u_1)=u_1\quad\mbox{and}\quad\varphi_2(e_1)=e_1,\,\varphi_2(u_1)=\omega u_1,
\]
where $\omega$ is a primitive third root of unity in $\bF$.

Define $\tilde\iota_i(x)=\iota_i(\tau^i(x))$ for all $i=1,2,3$ and $x\in\calC$. Then \eqref{eq:Albertproduct} translates to:
\begin{equation*}
\begin{split}
&E_i\tilde\iota_i(x)=0,\quad E_{i+1}\tilde\iota_i(x)=\frac{1}{2}\tilde\iota_i(x)=E_{i+2}\tilde\iota_i(x),\\
&\tilde\iota_i(x)\tilde\iota_{i+1}(y)=\tilde\iota_{i+2}(x*y),\quad
\tilde\iota_i(x)\tilde\iota_i(y)=2n(x,y)(E_{i+1}+E_{i+2}),
\end{split}
\end{equation*}
for $i=1,2,3$ and $x,y\in\calC$.

The commuting order $3$ automorphisms $\varphi_1$, $\varphi_2$ of $(\calC,*)$ extend to commuting order $3$ automorphisms of $\calA$ (which will be denoted by the same symbols) as follows: 
\begin{equation}\label{eq:vphi_1_2}
\varphi_j(E_i)=E_i,\quad \varphi_j\bigl(\tilde\iota_i(x)\bigr)=\tilde\iota_i(\varphi_j(x))\quad\mbox{for all}\quad x\in\calC\quad\mbox{and}\quad i=1,2,3,  
\end{equation}
where $j=1,2$. On the other hand, the linear map $\varphi_3\colon\cA\to\cA$ defined by
\begin{equation}\label{eq:vphi_3}
\varphi_3(E_i)=E_{i+1},\quad \varphi_3\bigl(\tilde\iota_i(x)\bigr)=\tilde\iota_{i+1}(x)\quad\mbox{for all}\quad x\in\calC\quad\mbox{and}\quad i=1,2,3, 
\end{equation}
is another order $3$ automorphism, which commutes with $\varphi_1$ and $\varphi_2$. The subgroup of $\Aut\calA$ generated by $\varphi_1,\varphi_2,\varphi_3$ is isomorphic to $\bZ_3^3$ and induces a $\ZZ_3^3$-grading on $\calA$ whose components are all $1$-dimensional. This grading is obviously fine, and $\bZ_3^3$ is its universal group. The nonzero homogeneous elements are invertible (in the Jordan sense), and any three of them whose degrees are independent in $\ZZ_3^3$ generate the Albert algebra (see \cite[\S 7.3]{Albert}).

The (unique up to equivalence) grading induced by $\varphi_1,\varphi_2,\varphi_3$ will be referred to as the {\em $\bZ_3^3$-grading} on the Albert algebra.

\smallskip

The classification of fine gradings on the Albert algebra was obtained in \cite{DM_f4} for $\chr{\FF}=0$ and in \cite{Albert} for $\chr{\FF}\ne 2$.

\begin{theorem}\label{th:finegrAlbert}
Let $\Gamma$ be a fine abelian group grading on the Albert algebra $\cA$ over an algebraically closed field $\FF$ of characteristic different from $2$.  
Then $\Gamma$ is equivalent either to the Cartan grading, the $\bZ_2^5$-grading, the $\bZ\times\bZ_2^3$-grading or the $\bZ_3^3$-grading. The last grading does not occur if $\chr{\FF}=3$.\hfill{$\square$}
\end{theorem}

\subsection{Cartan grading}\label{ss:WeylCartan}

The universal group is $\bZ^4$, which is contained in $E=\bR^4$. Consider the following elements of $\bZ^4$: 
\[
\begin{split}
\epsilon_0&=\degree\iota_1(e_1)=a_1=(1,0,0,0),\\
\epsilon_1&=\degree\iota_1(u_1)=g_1=(0,0,1,0),\\
\epsilon_2&=\degree\iota_1(u_2)=a_3+g_2=(-1,-1,0,1),\\
\epsilon_3&=\degree\iota_1(u_3)=-a_2+g_3=(0,-1,-1,-1).
\end{split}
\]
Note that $\epsilon_i$, $0\leq i\leq 3$, are linearly independent, but do not form a basis of $\bZ^4$. We have for instance $\degree\iota_2(e_1)=a_2=\frac{1}{2}(-\epsilon_0-\epsilon_1-\epsilon_2-\epsilon_3)$ and $\degree\iota_3(e_1)=\frac{1}{2}(-\epsilon_0+\epsilon_1+\epsilon_2+\epsilon_3)$. Moreover, the supports of the Cartan grading $\Gamma$ on each of the subspaces $\iota_i(\cC)$ are:
\[
\begin{split}
\Supp\iota_1(\cC)&=\{\pm \epsilon_i\;|\; 0\leq i\leq 3\},\\[6pt]
\Supp\iota_2(\cC)
 &=\Supp \iota_1(\cC)(\iota_3(e_1)+\iota_3(e_2))\\
 &=\bigl\{\frac{1}{2}(\pm\epsilon_0\pm\epsilon_1\pm\epsilon_2\pm\epsilon_3)\;|\;\text{even number of $+$ signs}\bigr\},\\[6pt]
\Supp\iota_3(\cC)
 &=\Supp \iota_1(\cC)(\iota_2(e_1)+\iota_2(e_2))\\
 &=\bigl\{\frac{1}{2}(\pm\epsilon_0\pm\epsilon_1\pm\epsilon_2\pm\epsilon_3)\;|\;\text{odd number of $+$ signs}\bigr\}.
\end{split}
\]
Let us consider the following subset $\Phi$ of $E$: 
\[
\begin{split}
\Phi&\bydef\Bigl(\Supp\Gamma\cup \{\alpha+\beta\;|\; \alpha,\beta\in\Supp\iota_1(\cC),\ \alpha\ne\pm\beta\}\Bigr)\setminus\{0\}\\
    &=\Supp\iota_1(\cC)\cup \Supp\iota_2(\cC)\cup \Supp\iota_3(\cC)\cup \{\pm\epsilon_i\pm\epsilon_j\;|\; 0\leq i\ne j\leq 3\},
\end{split}
\]
which is the root system of type $F_4$ (see \cite[Chapter VI.4.9]{Bourbaki}). Note that $\epsilon_i$, $i=0,1,2,3$, form an orthogonal basis of $E$ relative to the unique (up to scalar) inner product that is invariant under the Weyl group of $\Phi$.

Identifying the Weyl group $\W(\Gamma)$ with a subgroup of $\Aut(\bZ^4)$, and this with a subgroup of $GL(E)$, we have:
\[
\begin{split}
\W(\Gamma)&\subset \{\mu\in \Aut(\bZ^4)\;|\; \mu(\Supp\Gamma)=\Supp\Gamma\}\\
          &\subset \{\mu\in GL(E)\;|\; \mu(\Phi)=\Phi\}\defby \Aut\Phi.
\end{split}
\]
The latter group is the automorphism group of the root system $\Phi$, which coincides with its Weyl group.

\smallskip

If $\chr{\FF}\ne 2,3$, then we can work with the Lie algebra $\Der(\cA)$ and prove the next theorem using known results on the simple Lie algebra of type $F_4$ in \cite{Seligman}. The proof below works directly with the Cartan grading on the Albert algebra and is valid for $\chr{\FF}\ne 2$.

\begin{theorem}\label{th:WeylCartan}
Let $\Gamma$ be the Cartan grading on the Albert algebra over an algebraically closed field of characteristic different from $2$. Identify $\Supp\Gamma\setminus\{0\}$ with the short roots in the root system $\Phi$ of type $F_4$. Then $\W(\Gamma)=\Aut\Phi$.
\end{theorem}

\begin{proof}
Let us obtain first some distinguished elements in $\W(\Gamma)$.

1) The order $3$ automorphism of $\cA$: $\psi_{(123)}\colon E_i\mapsto E_{i+1}$, $\iota_i(x)\mapsto \iota_{i+1}(x)$, is in $\Aut(\Gamma)$, and its projection $\mu_{(123)}$ into $\W(\Gamma)$ permutes cyclically $\Supp\iota_1(\cC)$, $\Supp\iota_2(\cC)$ and $\Supp\iota_3(\cC)$.

2) The order $2$ automorphism of $\cA$: $\psi_{(23)}: E_1\mapsto E_1$, $E_2\leftrightarrow E_3$, $\iota_1(x)\mapsto \iota_1(\bar x)$, $\iota_2(x)\leftrightarrow \iota_3(\bar x)$, also belongs to $\Aut(\Gamma)$, and its projection $\mu_{(23)}$ sends $\epsilon_0$ to $-\epsilon_0$ (as $\psi_{(23)}(\iota_1(e_1))=\iota_1(e_2)$), and leaves invariant each $\epsilon_i$ for $i=1,2,3$.

3) Now consider $\Cl(\cC,n)$, the Clifford algebra of $\cC$ (regarded as a quadratic space), and the group $\Spin(\cC,n)\subset\Cl(\cC,n)$. It is well-known that any $c\in\Spin(\cC,n)$ gives rise to an automorphism $\psi_c$ of the Albert algebra fixing the idempotents $E_1,E_2,E_3$ --- see e.g. \cite{KMRT}. Explicitly, $\psi_c$ is defined by setting $\psi_c(\iota_1(z))=\iota_1(\chi_c(z))$, $\psi_c(\iota_2(z))=\iota_2(\rho^+_c(z))$ and $\psi_c(\iota_3(z))=\iota_3(\rho^-_c(z))$ for all $z\in\cC$, where $\chi_c(z)=c\cdot z\cdot c^{-1}$ and $\cdot$ denotes multiplication in the Clifford algebra. If $c=x\cdot y$ with $x,y\in\cC$, then 
\begin{equation}\label{eq:rho_pm}
\rho^+_c(z)=(zy)\bar{x}\quad\mbox{and}\quad\rho^-_c(z)=\bar{x}(yz)\quad\mbox{for all}\quad z\in\cC. 
\end{equation}
Consider the elements
\[
x=\frac{1}{\sqrt{2}}(e_1+e_2+u_1+v_1)\quad\mbox{and}\quad y=\frac{\bi}{\sqrt{2}}(e_1-e_2+u_1-v_1).
\]
Then $c=x\cdot y\in\Spin(\cC,n)$, since $n(x)=n(y)=1$. Also note that $c^{\cdot 2}=-1$, as $x$ and $y$ are orthogonal, and so $\chi_c$ has order $2$. Since $e_1+e_2$ and $y$ are also ortogonal, we compute:
\[
\begin{split}
\chi_c(e_1+e_2)&=c\cdot(e_1+e_2)\cdot c^{-1}=-c\cdot(e_1+e_2)\cdot c=-x\cdot(e_1+e_2)\cdot x\cdot y^{\cdot 2}\\
    &=-\frac{1}{2}(1-(e_1+e_2)\cdot(u_1+v_1))\cdot(e_1+e_2+u_1+v_1)\\
    &=-\frac{1}{2}\bigl(e_1+e_2+(e_1+e_2)^{\cdot 2}\cdot(u_1+v_1)+(u_1+v_1)-(e_1+e_2)\bigr)\\
    &=-(u_1+v_1).
\end{split}
\]
A similar calculation shows that $\chi_c(e_1-e_2)=-(u_1-v_1)$. Hence we have:
\[
\chi_c: e_1\leftrightarrow -u_1,\ e_2\leftrightarrow -v_1,\ u_2\mapsto u_2,\ u_3\mapsto u_3,\ v_2\mapsto v_2,\ v_3\mapsto v_3.
\]
Let us check that the associated automorphism $\psi_c$ of $\cA$ is in $\Aut(\Gamma)$. Since the action of $\psi_c$ on $\iota_1(\cC)$ is given by $\chi_c$, we already know that $\psi_c$ permutes the homogeneous components of $\iota_1(\cC)$. Since every homogeneous element of $\iota_3(\cC)$ lies either in $\iota_1(\cC)\iota_2(e_1)$ or in $\iota_1(\cC)\iota_2(e_2)$, and every homogeneous element of $\iota_2(\cC)$ lies either in $\iota_1(\cC)\iota_3(e_1)$ or in $\iota_1(\cC)\iota_3(e_2)$, it remains to check that $\psi_c(\iota_i(e_j))$ is homogeneous for $i=2,3$ and $j=1,2$. Using \eqref{eq:rho_pm}, we compute:
\[
\begin{split}
\psi_c(\iota_2(e_1))&=\iota_2(\rho^+_c(e_1))=\iota_2\bigl((e_1y)\bar x\bigr)\\
&=\frac{\bi}{2}\iota_2\bigl((e_1(e_1-e_2+u_1-v_1))(e_1+e_2-u_1-v_1)\bigr)\\
&=\frac{\bi}{2}\iota_2\bigl((e_1+u_1)(e_1+e_2-u_1-v_1)\bigr)=\frac{\bi}{2}\iota_2(e_1-u_1+u_1+e_1)=\bi\iota_2(e_1).
\end{split}
\]
Similar calculations show that $\psi_c(\iota_2(e_2))=-\bi\iota_2(e_2)$, $\psi_c(\iota_3(e_1))=-\bi\iota_3(v_1)$, and $\psi_c(\iota_3(e_2))=\bi\iota_3(u_1)$.
Therefore, $\psi_c\in\Aut(\Gamma)$, and its projection $\mu_c$ into $\W(\Gamma)$ acts as follows: $\epsilon_0\leftrightarrow \epsilon_1$, $\epsilon_2\mapsto \epsilon_2$, $\epsilon_3\mapsto \epsilon_3$.

4) Finally, the order $3$ automorphism $\tau$ of $\cC$ given by \eqref{eq:tauC} extends to an automorphism of $\cA$ fixing $E_i$ via  $\iota_i(x)\mapsto\iota_i(\tau(x))$ for all $x\in\cC$ and $i=1,2,3$. The projection of this automorphism into $\W(\Gamma)$ is the $3$-cycle $\epsilon_1\mapsto \epsilon_2\mapsto \epsilon_3\mapsto \epsilon_1$.

\smallskip

Now we are ready to prove the theorem. Any $\mu\in\Aut\Phi$ permutes the subsets $\Supp\iota_i(\cC)$, as these are the only subsets $S$ of $\Supp\Gamma\setminus\{0\}$ such that for any $\delta \in S$, $S=\{\pm\delta\}\cup\{\gamma\in\Supp\Gamma\setminus\{0\}\;|\;(\gamma,\delta)=0\}$. Thus, composing with a suitable power of $\mu_{(123)}$ from 1), we may assume $\mu(\Supp\iota_1(\cC))=\Supp\iota_1(\cC)$. But the group $\{\mu\in\Aut\Phi\;|\; \mu(\Supp\iota_1(\cC))=\Supp\iota_1(\cC)\}$ is isomorphic to $\ZZ_2^4\rtimes\sg(4)$, consisting of the permutations of the $\epsilon_i$'s followed by multiplication of some of the $\epsilon_i$'s by $-1$. This subgroup is generated by $\mu_{(23)}$ from 2), the transposition $\mu_c$ from 3), and the $3$-cycle from 4). 
\end{proof}

\begin{remark}
We have $\Stab(\Gamma)=\Diag(\Gamma)$. It is a maximal torus in the algebraic group $\Aut(\cA)$.
\end{remark}


\subsection{$\bZ_2^5$-grading}\label{ss:WeylZ25}

Write $\ZZ_2^5=\ZZ_2^2\times\ZZ_2^3$ where $\ZZ_2^3$ is generated by $c_j$, $j=1,2,3$, as in \eqref{CD_grading}. Then the $\bZ_2^5$-grading $\Gamma$ is defined by setting
\begin{equation}\label{Z_25_grading}
\degree\iota_1(1)=a,\quad\degree\iota_2(1)=b,\quad\degree\iota_3(w_j)=a+b+c_j,\; j=1,2,3,
\end{equation}
where $\{a,b\}$ is the standard basis of $\ZZ_2^2$.

\begin{theorem}\label{th:WeylZ25}
Let $\Gamma$ be the $\ZZ_2^5$-grading on the Albert algebra as in \eqref{Z_25_grading} over an algebraically closed field of characteristic different from $2$. 
Let $T$ be the subgroup of $\ZZ_2^5$ generated by $c_j$, $j=1,2,3$. Then $\W(\Gamma)=\{\mu\in\Aut(\ZZ_2^5)\;|\; \mu(T)=T\}$.
\end{theorem}

\begin{proof}
Let $K$ the subgroup generated by $a$ and $b$. Then $\ZZ_2^5=K\times T$. Identifying $\Aut(\ZZ_2^5)$ with $GL_5(2)$, the stabilizer of $T$ (as a set) consists of all matrices of the form $\left(\begin{tabular}{c|c} $*$&$0$\\ \hline $*$& $*$\end{tabular}\right)$.
The automorphism $\psi_{(123)}$ defined in step 1) in the proof of Theorem \ref{th:WeylCartan} and the automorphism $\psi_{(12)}$ that is analogous to $\psi_{(23)}$ defined in step 2) belong to $\Aut(\Gamma)$, and their projections into $\W(\Gamma)$ act as follows:
\[
\begin{split}
\mu_{(123)}&\colon\; a\mapsto b\mapsto a+b\mapsto a,\; c_j\mapsto c_j,\; j=1,2,3,\\
\mu_{(12)}&\colon\; a\leftrightarrow b,\; c_j\mapsto c_j,\; j=1,2,3.
\end{split}
\]
Therefore, the subgroup $\{\mu\in GL_5(2)\;|\; \mu(K)=K,\, \mu|_T=\id\}$, which consists of matrices of the form 
$\left(\begin{tabular}{c|c} $*$&$0$\\ \hline $0$& $I$\end{tabular}\right)$, is contained in $\W(\Gamma)$. (Note that this subgroup is the symmetric group on the elements $a$, $b$ and $a+b$.)

Now, consider the subgroup $\{\mu\in GL_5(2)\;|\; \mu(T)=T,\, \mu|_K=\id\}$, which consists of all matrices of the form 
$\left(\begin{tabular}{c|c} $I$&$0$\\ \hline $0$& $*$\end{tabular}\right)$. By Theorem \ref{th:Z23Weyl}, for any such $\mu$, there is an automorphism $\varphi$ of $\cC$ that belongs to $\Aut(\Gamma_0)$, where $\Gamma_0$ is the $\ZZ_2^3$-grading \eqref{CD_grading} on $\cC$, such that the projection of $\vphi$ into $\W(\Gamma_0)$ coincides with  $\mu|_T$. Then the automorphism of $\cA$ that fixes $E_i$ and takes $\iota_i(x)$ to $\iota_i(\varphi(x))$, for all $x\in \cC$ and $i=1,2,3$, belongs to $\Aut(\Gamma)$, and its projection into $\W(\Gamma)$ coincides with $\mu$. Hence the subgroup under consideration is contained in $\W(\Gamma)$.

For any $h\in T$, consider the element $\mu\in GL_5(2)$ such that $\mu(a)=a$, $\mu(b)=b+h$, and $\mu|_T=\id$. We claim that $\mu$ is in $\W(\Gamma)$. Take $x=1$ and $y$ a homogeneous element in $\cC_h$ of norm $1$. Then the element $c=x\cdot y\in\Spin(\cC,n)$ gives rise to the automorphism $\psi_c$ of $\cA$ --- see step 3) in the proof of Theorem \ref{th:WeylCartan}. The restriction $\psi_c\vert_{\iota_1(\cC)}$ is given by $\chi_c$ and hence stabilizes the homogeneous components in $\iota_1(\cC)$. For $z\in\cC_u$ with $u\in T$, we have $\iota_2(z)\in\cA_{b+u}$ and, using \eqref{eq:rho_pm}, we obtain $\psi_c(\iota_2(z))=\iota_2(\rho_c^+(z))=\iota_2(zy)\in\cA_{b+u+h}$; also, $\iota_3(z)\in\cA_{a+b+u}$ and hence we obtain $\psi_c(\iota_3(z))=\iota_3(\rho_c^-(z))=\iota_3(yz)\in\cA_{a+b+u+h}$. Therefore, $\psi_c\in\Aut(\Gamma)$, and its projection $\mu_c$ into $\W(\Gamma)$ fixes $a$ and the elements of $T$ while taking $b$ to $b+h$. By symmetry, we may also find an element in $\W(\Gamma)$ which fixes $b$ and the elements of $T$ and takes $a$ to $a+h$. Hence all matrices of the form $\left(\begin{tabular}{c|c} $I$&$0$\\ \hline $*$& $I$\end{tabular}\right)$ are contained in $\W(\Gamma)$.

So far, we have proved that the stabilizer of $T$ is contained in $\W(\Gamma)$. But conversely, if $\psi\in\Aut(\Gamma)$, then $\psi(\cA_e)=\cA_e$, so  $\psi$ permutes the idempotents $E_i$, $i=1,2,3$, and hence induces a permutation of the elements $a,b,a+b$. By composing $\psi$ with a suitable element of the stabilizer of $T$,  we may assume that $\psi(E_i)=E_i$ for all $i=1,2,3$. Then the projection of $\psi$ into $\W(\Gamma)$ preserves the cosets $a+T$, $b+T$ and $a+b+T$, and hence it preserves $T$.
\end{proof}

\begin{remark} 
As any $\psi\in\Stab(\Gamma)$ fixes $E_i$ and multiplies each $\iota_i(w_j)$, $i,j=1,2,3$, by either $1$ or $-1$, we see that $\Stab(\Gamma)=\Diag(\Gamma)$ is isomorphic to $\bZ_2^5$. 
\end{remark}

\subsection{$\bZ\times\bZ_2^3$-grading}\label{ss:WeylZZ23}

\begin{theorem}\label{th:WeylZZ23}
Let $\Gamma$ be the $\ZZ\times\ZZ_2^3$-grading on the Albert algebra defined by \eqref{eq:Z_Z23_grading} and \eqref{CD_grading} over an algebraically closed field of characteristic different from $2$. Then $\W(\Gamma)=\Aut(\bZ\times\bZ_2^3)$.
\end{theorem}

\begin{proof}
Let $T$ be the subgroup generated by $c_j$, $j=1,2,3$. Then $T$ is the torsion subgroup of $\bZ\times\bZ_2^3$ and hence we have $\mu(T)=T$ for all $\mu\in\Aut(\ZZ\times\ZZ_2^3)$. Let $a=(1,\bar 0,\bar 0,\bar 0)$.

The group $\Aut(\bZ\times \bZ_2^3)$ is generated by 1) the automorphism $\mu_0$ that fixes $T$ point-wise and takes $a$ to $-a$, 2) the automorphisms $\beta_h$, for $h\in T$, that fix $T$ point-wise and take $a$ to $a+h$, and 3) the automorphisms of $T$ extended to $\bZ\times \bZ_2^3$ by fixing $a$. We will show that all these automorphisms are contained in $\Aut(\Gamma)$.

1) The order $2$ automorphism $\psi_0$ of $\cA$ given by $S^{\pm}\mapsto S^{\mp}$, $\nu_{\pm}(x)\mapsto\nu_{\mp}(x)$, $\nu(a)\mapsto -\nu(a)$, for $x\in \cC$ and $a\in\cC_0$, belongs to $\Aut(\Gamma)$, and its projection into $\W(\Gamma)$ is precisely $\mu_0$.

2) Consider the $T$-grading $\Gamma_0$ on $\cC$ given by \eqref{CD_grading}. Fix $h\in H$. Pick a norm $1$ element $x\in \cC_h$, then take $y\in\cC_0$ homogeneous of norm $1$ with $n(x,y)=0$. Then $x=-xy^2=-(xy)y=zy$, where $z=-xy$ is a homogeneous element of norm $1$ in $\cC_0$. Note that $\degree y+\degree z=\degree x= h$. Consider the element $c=z\cdot y\in\Spin(\cC_0,n)$ and the associated automorphism $\psi_c$ of $\cA$. Then $\psi_c$ stabilizes the homogeneous components in $\nu(\cC_0)$, while $\psi_c(\nu_\pm(w))=\nu_\pm(\rho_c^+(w))=-\nu_\pm((wy)z))$ --- see \cite[Remark 6.4]{Albert}. Hence $\psi_c$ belongs to $\Aut(\Gamma)$, and its projection into $\W(\Gamma)$ fixes $T$ point-wise and takes $a+u$ to $a+u+h$ for any $u\in T$. Thus, this projection is the desired element $\beta_h$.

3) Given any automorphism $\mu$ of $T$, Theorem \ref{th:Z23Weyl} tells us that there is an automorphism $\varphi$ of $\cC$ that belongs to $\Aut(\Gamma_0)$ and whose projection into $\W(\Gamma_0)$ is $\mu$. The automorphism $\psi$ of $\cA$ determined by $\psi(S^\pm)=S^\pm$, $\psi(\nu_\pm(x))=\nu_\pm(\varphi(x))$, for all $x\in\cC$, belongs to $\Aut(\Gamma$), and its projection into $\W(\Gamma)$ is the automorphism of $\bZ\times\bZ_2^3$ fixing $a$ and restricting to $\mu$ on $T$.
\end{proof}

\begin{remark} 
One can show that $\Stab(\Gamma)=\Diag(\Gamma)$, which is isomorphic to $\FF^\times\times\bZ_2^3$. 
\end{remark}

\subsection{$\bZ_3^3$-grading}\label{ss:WeylZ33}

Recall that this grading occurs only if $\chr{\FF}\ne 3$. Let $\Gamma:\cA=\bigoplus_{g\in \ZZ_3^3}\cA_g$ be the grading induced by the commuting order $3$ automorphisms $\varphi_i$, $i=1,2,3$, defined by \eqref{eq:vphi_1_2} and \eqref{eq:vphi_3}, i.e.,
\begin{equation}\label{eq:Z33_grading}
\vphi_i(X)=\omega^{k_i}X\quad\mbox{for all}\quad X\in\cA_{(\bar k_1,\bar k_2,\bar k_3)},
\end{equation} 
where $\omega$ is a fixed primitive third root of unity. Let $\{g_1,g_2,g_3\}$ be a basis of $\ZZ_3^3$ and pick nonzero $X_i\in\cA_{g_i}$, $i=1,2,3$. Then $X_1,X_2,X_3$ generate the Albert algebra, and we can scale them so that $X_i^3=1$ and hence $N(X_i)=1$, $i=1,2,3$.    

The subalgebra generated by $X_3$ is isomorphic to $\FF\times\FF\times\FF$, so there exists an automorphism of $\cA$ sending it to $\bF E_1\oplus\bF E_2\oplus\bF E_3$. Permuting $E_1,E_2,E_3$ if necessary, we may assume that $\varphi_3(E_i)=E_{i+1}$. In other words, we may assume $X_3=\sum_{i=1}^3\omega^{-i}E_i$. The subalgebra fixed by $\varphi_3$ is $\bF 1\oplus\{\sum_{i=1}^3\tilde\iota(x)\;|\; x\in\cC\}$, so there are elements $x,y\in \cC$ such that $X_1=\frac{1}{2}\sum_{i=1}^3\tilde\iota(x)$ and $X_2=\frac{1}{2}\sum_{i=1}^3\tilde\iota(y)$.

For any $z\in\cC$, the norm of $\sum_{i=1}^3\tilde\iota_i(z)$ is given by:
\[
\begin{split}
N\bigl(\tilde\iota_1(z)+\tilde\iota_2(z)+\tilde\iota_3(z)\bigr)
    &=N\bigl(\iota_1(\tau(z))+\iota_2(\tau^2(z))+\iota_3(z)\bigr)\\
    &=8n(\tau(z),\overline{\tau^2(z)}\,\bar z)=8n(z,\overline{\tau(z)}\,\overline{\tau^2(z)})\\
    &=8n(z,z*z).
\end{split}
\]
Since $N(X_i)=1$ for $i=1,2,3$, we get $n(x,x*x)=1=n(y,y*y)$. And since $T(X_i^2)=0$, we have $n(x)=0=n(y)$. Also, since the homogeneous components $\cA_{(\pm\bar 1,\bar 0,\bar 0)}$ and $\cA_{(\bar 0,\pm\bar 1,\bar 0)}$ are orthogonal relative to the trace form $T$, we conclude that the subspaces $\bF x\oplus\bF x*x$ and $\bF y\oplus \bF y*y$ are orthogonal relative to the norm of $\cC$. Now \cite[Proposition 3.9 and Theorem 3.12]{ElduqueGrSym} show that either $x*y=0$ or $y*x=0$, but not both, and that $x,y$ generate the Okubo algebra $(\cC,*,n)$ with multiplication table independent of $x$ and $y$.

Fix $a,b\in \cC$ with $a*b=0$, $n(a)=0=n(b)$, $n(a,a*a)=1=n(b,b*b)$, and $n\bigl(\bF a+\bF a*a,\bF b+\bF b*b\bigr)=0$. Let $\Gamma^+$ and $\Gamma^-$ be two $\ZZ_3^3$-gradings on the Albert algebra that are determined by the following conditions:
\begin{equation}\label{eq:two_Z33_gradings}
\begin{array}{cc}
\Gamma^+: & \Gamma^-:\\
\begin{array}{rcl}
\degree\bigl(\sum_{i=1}^3\tilde\iota_i(a)\bigr)&=&(\bar 1,\bar 0,\bar 0),\\[2pt] 
\degree\bigl(\sum_{i=1}^3\tilde\iota_i(b)\bigr)&=&(\bar 0,\bar 1,\bar 0),\\[2pt]
\degree\bigl(\sum_{i=1}^3\omega^{-i}E_i\bigr)&=&(\bar 0,\bar 0,\bar 1),
\end{array} &
\begin{array}{rcl} 
\degree\bigl(\sum_{i=1}^3\tilde\iota_i(a)\bigr)&=&(\bar 0,\bar 1,\bar 0),\\[2pt]
\degree\bigl(\sum_{i=1}^3\tilde\iota_i(b)\bigr)&=&(\bar 1,\bar 0,\bar 0),\\[2pt]
\degree\bigl(\sum_{i=1}^3\omega^{-i}E_i\bigr)&=&(\bar 0,\bar 0,\bar 1).
\end{array}
\end{array}
\end{equation}
For example, we may take $a=e_1$ and $b=u_1$. Then $\Gamma^+=\Gamma$.

Note that for $\Gamma^+$ with $\{g_1,g_2,g_3\}$ being the standard basis of $\ZZ_3^3$, we have:
\[
\begin{split}
(X_1X_2)X_3&=\Bigl(\frac{1}{4}\sum_{i=1}^3\tilde\iota_i(b*a)\Bigr)X_3\\
    &=\frac{1}{8}\sum_{i=1}^3(\omega^{-(i+1)}+\omega^{-(i+2)})\tilde\iota_i(b*a)\\
    &=-\frac{1}{8}\sum_{i=1}^3\omega^{-i}\tilde\iota_i(b*a),
\end{split}
\]
while
\[
\begin{split}
X_1(X_2X_3)
    &=X_1\Bigl(\frac{1}{4}\sum_{i=1}^3(\omega^{-(i+1)}+\omega^{-(i+2)})\tilde\iota_i(y)\Bigr)\\
    &=X_1\Bigl(-\frac{1}{4}\sum_{i=1}^3\omega^{-i}\tilde\iota_i(y)\Bigr)\\
    &=-\frac{1}{8}\bigl(\omega^{-2}\tilde\iota_1(b*a)+\omega^{-3}\tilde\iota_2(b*a)
        +\omega^{-1}\tilde\iota_3(b*a)\bigr),
\end{split}
\]
so that $(X_1X_2)X_3=\omega X_1(X_2X_3)$. However, for $\Gamma^-$ with $\{g_1,g_2,g_3\}$ being the standard basis of $\ZZ_3^3$, analogous computations give $(X_1X_2)X_3=\omega^{-1}X_1(X_2X_3)$. Therefore $\Gamma^+$ and $\Gamma^-$ are not isomorphic.

\begin{theorem}\label{th:WeylZ33}
Let $\Gamma$ be the $\ZZ_3^3$-grading on the Albert algebra as in \eqref{eq:Z33_grading} over an algebraically closed field of characteristic different from $2$ and $3$. 
Then $\W(\Gamma)$ is the commutator subgroup of $\Aut(\ZZ_3^3)$, i.e., $\W(\Gamma)\cong SL_3(3)$.
\end{theorem}

\begin{proof}
We may assume $\Gamma=\Gamma^+$ as in \eqref{eq:two_Z33_gradings}. Identify $\Aut(\ZZ_3^3)$ with $GL_3(3)$. For any $\mu\in GL_3(3)$, let $g_j$, $j=1,2,3$, be the images of the elements of the standard basis (i.e., the columns of matrix $\mu$). Pick elements $X'_j$ such that $\deg X'_j=g_j$ and $(X'_j)^3=1$, $j=1,2,3$. We have shown that there exists an automorphism of $\cA$ either sending $X'_j$ to $X_j$ associated with $\Gamma^+$, $j=1,2,3$, or sending $X'_j$ to $X_j$ associated with $\Gamma^-$, $j=1,2,3$, but not both. This shows that $\Aut(\Gamma)$ has index $2$ in $GL_3(3)$. Since the commutator subgroup of $GL_3(3)$ is $SL_3(3)$, and it has index $2$, we conclude that $\W(\Gamma)=SL_3(3)$.
\end{proof}

\begin{remark}
Clearly, $\Stab(\Gamma)=\Diag(\Gamma)$ is isomorphic to $\ZZ_3^3$.
\end{remark}



\end{document}